\documentclass[reqno,english]{amsart}
\usepackage{amsfonts,amsmath,latexsym,verbatim,amscd,mathrsfs,color,array}

\usepackage{amsmath,amssymb,amsthm,amsfonts,graphicx,color}
\usepackage{amssymb}
\usepackage{pdfsync}
\usepackage{epstopdf}
\usepackage[colorlinks=true]{hyperref}
\usepackage{subfigure}

\makeatletter
\def\@currentlabel{2.1}\label{e:dispaa}
\def\@currentlabel{2.21}\label{e:dispau}
\def\@currentlabel{2.22}\label{e:dispav}
\def\@currentlabel{2.23}\label{e:dispaw}
\def\@currentlabel{2.24}\label{e:dispax}
\def\theequation{\thesection.\@arabic\c@equation}
\makeatother
\let\oldbibliography\thebibliography
\renewcommand{\thebibliography}[1]{%
\oldbibliography{#1}%
\setlength{\itemsep}{0pt}%
}

\renewcommand{\theequation}{\thesection.\arabic{equation}}
\newtheorem{lemma}{Lemma}[section]

\newtheorem{proposition}{Proposition}[section]
\newtheorem{corollary}{Corollary}[section]
\newtheorem{remark}{Remark}[section]

\newtheorem{open problem}{Open Problem}[section]
\newtheorem{open question}{Open Quesion}[section]
\newcommand{\bremark}{\begin{remark} \em}
\newcommand{\eremark}{\end{remark} }

\newtheorem{numerical/experimental results}{Numerical/Experimental results}[section]

\newtheorem{theorem}{Theorem}[section]

\newcommand{\BE}{\begin{equation}}
\newcommand{\BEN}{\begin{equation*}}
\newcommand{\EE}{\end{equation}}
\newcommand{\EEN}{\end{equation*}}
\newcommand{\BL}{\begin{lemma}}
\newcommand{\EL}{\end{lemma}}
\newcommand{\BT}{\begin{theorem}}
\newcommand{\ET}{\end{theorem}}
\newcommand{\BP}{\begin{proposition}}
\newcommand{\EP}{\end{proposition}}
\newcommand{\BC}{\begin{corollary}}
\newcommand{\EC}{\end{corollary}}
\renewcommand{\Re}{\operatorname{Re}}
\renewcommand{\Im}{\operatorname{Im}}
\usepackage{amsmath}

\DeclareMathOperator*{\argmin}{arg\,min}

\begin{document}

%\title[Green Function]{Two-component reduced functional}
%%%%%%%%%%%%%%%%%%%%%%%%%%%%%%%%%%%%%%%%%%%%%%%%%%%%%%%%%%%%%%%%%%%%%%

\title{A new type of minimizers in lattice energy and its application}

\author{Kaixin Deng}

\author{Senping Luo}

\address[K.~Deng]{School of Mathematics and statistics, Jiangxi Normal University, Nanchang, 330022, China}
\address[S.~Luo]{School of Mathematics and statistics, Jiangxi Normal University, Nanchang, 330022, China}

\email[S.~Luo]{luosp1989@163.com}

\email[K.~Deng]{Dengkaikai1999@126.com}

\begin{abstract}
Let  $z\in \mathbb{H}:=\{z= x+ i y\in\mathbb{C}: y>0\}$  and
$$
\mathcal{K}(\alpha;z):=\sum_{ (m,n)\in \mathbb{Z} ^2 }\frac{{\left| mz+n \right|}^2}{{{\Im}(z)}}e^{-\pi\alpha\frac{ \left|mz+n\right|^2}{\Im(z)}}.
$$
 In this paper, we characterize the following minimization problem
\begin{equation}\aligned\nonumber
\min_{   \mathbb{H} } \big(\mathcal{K}(\alpha;z)-b\mathcal{K}(2\alpha;z)\big).
\endaligned\end{equation}

 We prove that there exist hexagonal to skinny-rhombic minimizers, which is a novel finding in the literature.

\end{abstract}

\maketitle

%\tableofcontents

%%%%%%%%%%%%%%%%%%%%%%%%%%%%%%%%%%%%%%%%%%%%%%%%%%%%%%%%%%%%%%%%%%%%%%%%%%%%%%%%%%%%%%%%%%%%%%%%%%%%%%%%%%%%%%%%%%%%%%%%%%%%%%%%%%%%%%%%%%%%%%%%
\section{Introduction and Statement of Main Results}
\setcounter{equation}{0}
{\it The fact that chemical elements combine to form crystals, periodic objects where the atoms are arranged in a periodic lattice of points with a limited set of symmetries, has been a basic belief for more than two centuries} (Bindi \cite{Bindi2020}). However, there is still an open largely crystal problem: what are the fundamental mechanisms behind the spontaneous arrangement of atoms into periodic configurations at low temperatures? (Radin \cite{Radin1987}). This famous problem, called as `Crystallization Conjecture', was proposed by Radin in 1987, a recent review of this conjecture can be found in Blanc-Lewin \cite{Blanc2015}.

 We are particularly interested in two-dimensional crystals. These systems capture many essential features of higher dimensions without the added complexity. Two-dimensional crystals play a crucial role in various physical systems, such as lipid monolayers on the surface of water \cite{Kaganer1999}, a monolayer of electrons on the surface of liquid helium, rare-gas clusters \cite{Schwerdtfeger2006}, colloidal systems in two dimensions \cite{Peeters1987}, and dusty plasmas \cite{Nosenko2004}. For a comprehensive overview of these applications, see B\'etermin \cite{Bet2015,Bet2016,Bet2018,BP2017,Betermin2021AHP}.

A fundamental mathematical and physical model for the Crystallization Conjecture and two-dimensional crystals is given by:
  \begin{equation}\aligned\label{EFL}
\min_L E_f(L ), \;\;\hbox{where}\;\;E_f(L):=\sum_{\mathbb{P}\in L \backslash\{0\}} f(|\mathbb{P}|^2),\; |\cdot|\;\hbox{is the Euclidean norm on}\;\mathbb{R}^2.
\endaligned\end{equation}

 The $E_f(L )$ denotes the lattice energy per particle of the crystals, the summation ranges over all the lattice points except the origin $0$.
 $f$ is the background potential of the system and $L$ denotes the lattice. For the Riesz potential $f(r^2)=|r|^{-2s}$ with $s>1$, the lattice energy $E_f(L)$ becomes the classical Epstein zeta function,
 Rankin \cite{Rakin1953}, Cassels \cite{Cassels1963}, Ennola \cite{Ennola1964} and Diananda \cite{Dia1964} proved that the hexagonal lattice is the unique minimizer of the Epstein zeta function up to the action by modular group. For the Gaussian potential $f(r^2)=e^{-\pi \alpha r^2}$ with $ \alpha>0$, the lattice energy becomes the classical theta function, and
 Montogmery \cite{Mon1988} proved that the hexagonal lattice still minimize the lattice energy. In fact, Montogmery's Theorem \cite{Mon1988} implies the results of Rankin \cite{Rakin1953}, Cassels \cite{Cassels1963}, Ennola \cite{Ennola1964} and Diananda \cite{Dia1964}. The motivation of these research come from pure interest of number theory, and it turns out these theorems have deep applications in lattice energy.

 In this paper, we are interested in the lattice energy of the following difference form:
   \begin{equation}\aligned\label{EFL1}
\min_L E_{f}(L ), \;\;\hbox{where}\;\;E_f(L):=E_{f_1}(L)-E_{f_2}(L),\; f=f_1-f_2,
\endaligned\end{equation}
 where the potentials $f_1, f_2$ are two potentials with simple forms. Mathematically, the problem \eqref{EFL1} is equivalent to the problem \eqref{EFL}, but physically, it introduces new insights. For more details on this model and its applications in physics, we refer to B\'etermin \cite{Bet2018,Betermin2021JPA,Betermin2021AHP,Betermin2021Arma}, B\'etermin-Petrache\cite{Bet2019AMP}, B\'etermin-Faulhuber-Kn$\ddot{u}$pfer \cite{Bet2020}, and B\'etermin-Friedrich-Stefanelli \cite{Betermin2021LMP}. The problem \eqref{EFL1} was initiated by the study of difference of Yukawa potential (B\'etermin \cite{Bet2016}),
 Lennard-Jones potential (B\'etermin \cite{Bet2018}), and Morse potential (B\'etermin \cite{Bet2019}).

In this paper, we consider the potential of the form:

\begin{equation}\aligned\label{Potential}
f=f_1-f_2,\;\;\hbox{where}\;f_1(r^2)=r^2e^{-\pi \alpha r^2},\;\;f_2(r^2)=br^2e^{-\pi 2\alpha r^2},\;\hbox{and}\;\alpha\geq2,\;b>0.
\endaligned\end{equation}

Let $ z\in \mathbb{H}:=\{z\in\mathbb{C}: \Im(z)>0\}$ and $\Lambda =\frac{1}{\sqrt{\Im(z)}}\big({\mathbb Z}\oplus z{\mathbb Z}\big)$ be the lattices with unit density in $ \mathbb{R}^2$. Define
\begin{equation}\aligned\label{define}
\mathcal{K}(\alpha;z):=\underset{\mathbb{P}\in \Lambda}{\sum}{\left|  \mathbb{P} \right|}^2{e}^{-\pi\alpha{\left|  \mathbb{P} \right|}^2}=\sum_{ (m,n)\in \mathbb{Z} ^2 }\frac{{\left| mz+n \right|}^2}{{{\Im}(z)}}e^{-\pi\alpha\frac{ \left|mz+n\right|^2}{\Im(z)}}.
\endaligned\end{equation}

Then the lattice energy \eqref{EFL1} under potential \eqref{Potential} becomes
\begin{equation}\aligned\nonumber
\mathcal{K}(\alpha;z)-b \mathcal{K}(2\alpha;z),\;\hbox{where}\;\alpha\geq2,\;b>0.
\endaligned\end{equation}
It is interesting to study the phase transitions of the lattice energy under the potential \eqref{Potential}.
To this end, we prove the following theorem:

\begin{theorem}\label{Th1} Assume that $ \alpha \geq 2$. Consider the lattice energy:
\begin{equation}\aligned\nonumber
\underset{z\in \mathbb{H}}{\min} \big(\mathcal{K}(\alpha;z)-b \mathcal{K}(2\alpha;z)\big).
\endaligned\end{equation}
Then, up to the action by modular group, there exists two thresholds $b_{c_1}<b_{c_2}$, where $b_{c_1}$ depends on
 $\alpha$
with bound $b_{c_1}\in(2,2\sqrt2)$ and $b_{c_2}=2\sqrt2$ is independent of $\alpha$. Specifically, the following holds:
\begin{itemize}
\item [(1)] For $b\leq b_{c_1}$, the minimizer is $e^{i\frac{\pi}{3}}$, corresponding to a hexagonal lattice.
\item [(2)] For $b_{c_1}<b<b_{c_2}$, the minimizer is $\frac{1}{2}+i y_{b}$ with $y_b>\frac{\sqrt3}{2}$, corresponding to a skinny-rhombic lattice. Furthermore, as $b$ approaches $2\sqrt{2}$, $y_{b}$ approaches $+\infty$.
\item [(3)] For $b\geq b_{c_2}$, the minimizer does not exist.
\end{itemize}
  \end{theorem}

B\'etermin discovered the hexagonal-rhombic-square-rectangular phase transitions of the system under Lennard-Jones potential and Morse potential
in \cite{Bet2018} and \cite{Bet2019}, respectively. A rigorous proof of these phase transitions can be found in tri-copolymer systems (Luo-Ren-Wei \cite{Luo2019}) and Bose-Einstein condensates (Luo-Wei \cite{Luo2022}), respectively.
In Theorem \ref{Th1}, we uncover a new type of phase transition from hexagonal to skinny-rhombic.
\begin{figure}
\centering
 \includegraphics[scale=0.45]{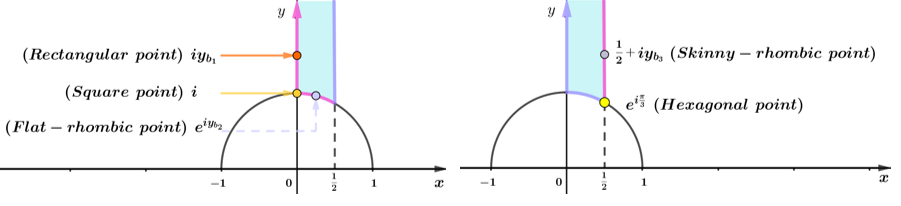}
 \caption{The existing minimizers and a new type of minimizers.}
 \label{d}
\end{figure}

\vskip0.1in

This paper is organized as follows. In Section 2, we introduce the symmetries of the function $\mathcal{K}(\alpha;z)$, such that this problem can be reduced to the fundamental domain. In Section 3, we prove that the minimization problem on the fundamental domain can be reduced to its right boundary. In Section 4, we explain partially the case where the minimizer is the hexagonal point and the methods can be applied to the relevant problems. Finally, in Section 5, we prove Theorem \ref{Th1}.

\section{Preliminaries }
In this section, we collect some basic properties of the theta function along with some useful estimates of Jacobi theta function. We also introduce a group related to the functional $\theta(\alpha;z)$. The generators of the group are given by
\begin{equation}\aligned\label{GroupG1}
\mathcal{G}: \hbox{the group generated by} \;\;\tau\mapsto -\frac{1}{\tau},\;\; \tau\mapsto \tau+1,\;\;\tau\mapsto -\overline{\tau}.
\endaligned\end{equation}

By Definition (\ref{GroupG1}), the fundamental domain associated to modular group $\mathcal{G}$ is
\begin{equation}\aligned\label{Fd1}
\mathcal{D}_{\mathcal{G}}:=\{
z\in\mathbb{H}: |z|>1,\; 0<x<\frac{1}{2}
\}.
\endaligned\end{equation}

The following lemma  characterizes the invariance properties of the theta function under the action of the group $\mathcal{G}$:
\begin{lemma}[B\'etermin \cite{Bet2018}, Luo-Wei \cite{Luo2022}]\label{G111} For any $\alpha>0$, any $\gamma\in \mathcal{G}$ and $z\in\mathbb{H}$,
$\theta (\alpha; \gamma(z))=\theta (\alpha; z)$.
\end{lemma}

$\mathcal{K}(\alpha;z)$ has a close connection to the theta function.

\begin{lemma}[The relationship between the function $\mathcal{K}(\alpha;z)$ and the theta function]\label{2lem2}
\begin{equation}\aligned\nonumber
\mathcal{K}(\alpha;z)=-\frac{1}{\pi} \frac{\partial}{\partial{\alpha}}\theta (\alpha;z).
\endaligned\end{equation}
\end{lemma}

\begin{lemma}[The invariance of $(\mathcal{K}(\alpha;z)-b \mathcal{K}(2\alpha;z))$]\label{2lem3} For any $\alpha>0,\:b\in\mathbb{R}$, any $\gamma\in \mathcal{G}$ and $z\in\mathbb{H}$, it holds that
$$ \mathcal{K}(\alpha;z)-b \mathcal{K}(2\alpha;z)=\mathcal{K}\big(\alpha;\gamma(z)\big)-b \mathcal{K}\big(2\alpha;\gamma(z)\big).$$
\end{lemma}
The Jacobi theta function is defined as follows:
\begin{equation}\aligned\label{Jacobi}\nonumber
\vartheta_J(z;\tau):=\sum_{n=-\infty}^\infty e^{i\pi n^2 \tau+2\pi i n z}.
 \endaligned\end{equation}
The classical one-dimensional theta function is given by
\begin{equation}\aligned\label{TXY}
\vartheta(X;Y):=\vartheta_J(Y;iX)=\sum_{n=-\infty}^\infty e^{-\pi n^2 X} e^{2n\pi i Y}.
 \endaligned\end{equation}
By the Poisson summation formula, it holds that
\begin{equation}\aligned\label{Poisson}
\vartheta(X;Y)=X^{-\frac{1}{2}}\sum_{n=-\infty}^\infty e^{-\pi \frac{(n-Y)^2}{X}} .
 \endaligned\end{equation}
To estimate bounds of quotients of derivatives of $\vartheta(X;Y)$, we denote that
\begin{equation}\aligned\nonumber
\underline\vartheta_{1}(X)&:=4\pi e^{-\pi X}\big(1-\mu(X)\big), \:\:\:\:\:\:\:\:\:\:\overline{\vartheta}_{1}(X):=4\pi e^{-\pi X}\big(1+\mu(X)\big),\:\:\:\:\:\:\\
\underline\vartheta_{2}(X)&:=\pi e^{-\frac{\pi}{4X}}X^{-\frac{3}{2}},\:\:\:\:\:\:\:\:\:\:\:\:\:\:\:\:\:\:\:\:\:\:\:\:\:\overline{\vartheta}_{2}(X):=X^{-\frac{3}{2}},\:\:\:\:\:\:\:\:\:\:\:\:\:\\
\endaligned\end{equation}
and
\begin{equation}\aligned\label{duct}
\mu(X)&:=\sum_{n=2}^{\infty}n^2e^{-\pi (n^2-1) X},\:\:\:\:\:\:\:\:\:\:\:\:\hat{\mu}(X):=\sum_{n=2}^{\infty}n^2e^{-\pi (n^2-1) X}(-1)^{n+1},\\
\nu(X)&:=\sum_{n=2}^{\infty}n^4e^{-\pi (n^2-1) X},\:\:\:\:\:\:\:\:\:\:\:\:\hat{\nu}(X):=\sum_{n=2}^{\infty}n^4e^{-\pi (n^2-1) X}(-1)^{n+1}.\\
    \endaligned\end{equation}

There are some useful estimates for quotients of 1-d theta functions and their derivatives.
\begin{lemma}[Luo-Wei \cite{Luo2022}]\label{2lem4} Assume that $\sin(2\pi Y)>0$. It holds that
\begin{itemize}
  \item [(1)] if $ X>\frac{1}{5}$, then $ -\overline{\vartheta}_{1}(X)\sin(2\pi Y)\leq\frac{\partial}{\partial Y}\vartheta(X;Y)\leq-\underline\vartheta_{1}(X)\sin(2\pi Y);$
  \item [(2)] if $X< \frac{\pi}{\pi+2} $, then $-\overline\vartheta_{2}(X)\sin(2\pi Y)\leq\frac{\partial}{\partial Y}\vartheta(X;Y)\leq-\underline\vartheta_{2}(X)\sin(2\pi Y).$
\end{itemize}
 \end{lemma}
%%%%%%%%%%%%%%%%%%%%%%%%%%%%%%%%%%%%%%%%%%%%%%%%%%%%%%%%%%%%%%%%%%%%%%%%%%%%%%%%%%%%%%%%%%%%%%%%%%%%%%%%%%%%%%%%%%%%%%%%%%%%%%%%%
\begin{remark}\label{2rem1}
By Lemma \ref{2lem4}, $-\frac{\partial}{\partial Y}\vartheta(X;Y)>0$ for $X\in\mathbb{R},Y\in(0,\frac{1}{2})$.
\end{remark}

Furthermore, the following estimates for higher order derivatives of quotients of 1-d theta functions hold:
\begin{lemma}[\cite{Luo2023,DK2024}]\label{2lem4add}
Assume that $Y>0,k\in \mathbb{N^{+}}$. It holds that
\begin{itemize}
  \item [(1)] for $ X>\frac{1}{5}$, then $\left|\frac{\vartheta_{Y}(X;kY)}{\vartheta_{Y}(X;Y)} \right|\leq k\cdot \frac{1+\mu(X)}{1-\mu(X)};$
  \item [(2)] for $X< \frac{\pi}{\pi+2} $, then $\left|\frac{\vartheta_{Y}(X;kY)}{\vartheta_{Y}(X;Y)} \right|\leq k\cdot \frac{1}{\pi}e^{\frac{\pi}{4X}};$

  \item [(3)] for $X\geq \frac{1}{5}$, then $-{\pi}\cdot \frac{1+{\nu}(X)}{1+{\mu}(X)}\leq \frac{\vartheta_{XY}(X;Y)}{\vartheta_{Y}(X;Y)} \leq -{\pi} \cdot\frac{1+\hat{\nu}(X)}{1+\hat{\mu}(X)};$
  \item [(4)] for $0<X\leq\frac{1}{2}$, then $\frac{\frac{3}{4}{X}^2+2{\pi}^2 e^{-\frac{\pi}{X}}}{-\frac{1}{2}{X}^3+2\pi{X}^2e^{-\frac{\pi}{X}}}\leq \frac{\vartheta_{XY}(X;Y)}{\vartheta_{Y}(X;Y)} \leq \frac{\pi}{4{X}^2};$
  \item [(5)] for $ X\geq \frac{1}{5}$, then $\left|\frac{\vartheta_{XY}(X;kY)}{\vartheta_{Y}(X;Y)} \right|\leq k\pi \cdot \frac{1+\nu(X)}{1-\mu(X)};$
  \item [(6)] for $0<X\leq \frac{9}{20}$, then $\left|\frac{\vartheta_{XY}(X;kY)}{\vartheta_{Y}(X;Y)} \right|\leq \frac{k}{4X^2}e^{\frac{\pi}{4X}}.$
\end{itemize}
 \end{lemma}

\section{The transversal monotonicity}

We define the vertical line $\Gamma_{c}$ in the upper half-plane $\mathbb{H}$ as follows:
$$\Gamma_{c}:=\{z\in\mathbb{H}: \Re(z)=\frac{1}{2},\; \Im(z)\geq\frac{\sqrt3}{2}\}.$$
By the group invariance (Lemma \ref{2lem3}), one has
\begin{equation}\aligned\label{T1}
\underset{z\in \mathbb{H}}{\min}(\mathcal{K}(\alpha;z)-b \mathcal{K}(2\alpha;z))=\underset{z\in \overline{\mathcal{D}_{\mathcal{G}}}}{\min}(\mathcal{K}(\alpha;z)-b \mathcal{K}(2\alpha;z)).
\endaligned\end{equation}

In this section, we aim to establish the following theorem:
\begin{theorem}\label{3Thm1}  Assume that $ \alpha \geq \frac{3}{2}$. Then for $b\leq 2\sqrt{2}$, it holds that
\begin{equation}\aligned\nonumber
\underset{z\in \mathbb{H}}{\min}\big(\mathcal{K}(\alpha;z)-b \mathcal{K}(2\alpha;z)\big)&=\underset{z\in \overline{\mathcal{D_{\mathcal{G}}}}}{\min}\big(\mathcal{K}(\alpha;z)-b \mathcal{K}(2\alpha;z)\big)\\
&=\underset{z\in \Gamma_{c}}{\min}\big(\mathcal{K}(\alpha;z)-b \mathcal{K}(2\alpha;z)\big).
\endaligned\end{equation}
\end{theorem}

To prove Theorem \ref{3Thm1}, we establish the following monotonicity result:

\begin{proposition}\label{Thm2}  Assume that $ \alpha \geq \frac{3}{2},\: b\leq 2\sqrt{2}$. Then it holds that
\begin{equation}\aligned\nonumber
\frac{\partial}{\partial{x}}\big(\mathcal{K}(\alpha;z)-b \mathcal{K}(2\alpha;z)\big)< 0,\:\: for\:\: z\in \mathcal{D_{G}}.
\endaligned\end{equation}
\end{proposition}

Note that Proposition \ref{Thm2} implies Theorem \ref{3Thm1}.
The rest of this section is devoted to proving Proposition \ref{Thm2}.

\subsection{The estimates}
  Using the ideas from B\'etermin \cite{Bet2016} and Luo-Wei \cite{LW2022,Luo2023,Luo2024+}, we will utilize the exponential expansion of the theta function, such that Theorem \ref{Thm2} is reduced to the study of Jacobi theta function.
\begin{lemma}[\cite{Luo2022,Mon1988}\label{3Lemma1}] We have the exponential expansion of $\theta (\alpha;z)$. For $\alpha, y>0$, it holds that
\begin{equation}\aligned\nonumber
\theta (\alpha;z)
&=\sqrt{\frac{y}{\alpha}}\sum_{n\in\mathbb{Z}}e^{-\alpha \pi y n^2}\vartheta(\frac{y}{\alpha};nx)\\
&=2\sqrt{\frac{y}{\alpha}}\sum_{n=1}^\infty e^{-\alpha \pi y n^2}\vartheta(\frac{y}{\alpha};nx)+\sqrt{\frac{y}{\alpha}}\vartheta(\frac{y}{\alpha};0).
\endaligned\end{equation}
\end{lemma}

Based on Lemmas \ref{2lem2} and \ref{3Lemma1}, we derive the exponential expansion of the function $(\mathcal{K}(\alpha;z)-b \mathcal{K}(2\alpha;z))$ in Lemma \ref{3Lemma2}. The proof of Lemma \ref{3Lemma2} is straightforward, hence we omit it.
\begin{lemma}\label{3Lemma2} We have the following expression of $(\mathcal{K}(\alpha;z)-b \mathcal{K}(2\alpha;z))$: for $\alpha,y>0$, it holds that
\begin{equation}\aligned\nonumber
\mathcal{K}(\alpha;z)-b \mathcal{K}(2\alpha;z)=&\frac{1}{\pi}{2}^{-\frac{5}{2}}\alpha^{-\frac{5}{2}}y^{\frac{1}{2}}\Big(
2\sqrt{2} \alpha\sum_{n\in\mathbb{Z}} e^{-\pi\alpha y n^2}\vartheta(\frac{y}{\alpha};nx)
-b \alpha \sum_{n\in\mathbb{Z}} e^{-2\pi\alpha  y n^2}\vartheta(\frac{y}{2 \alpha};nx)\\
&+4\sqrt{2}\pi \alpha^2y\sum_{n\in\mathbb{Z}} n^2e^{-\pi\alpha y n^2}\vartheta(\frac{y}{\alpha};nx)
-4 \pi b \alpha^2y\sum_{n\in\mathbb{Z}} n^2e^{-2\pi\alpha  y n^2}\vartheta(\frac{y}{2 \alpha};nx)\\
&+4\sqrt{2}y\sum_{n\in\mathbb{Z}} e^{-\pi\alpha y n^2}\vartheta_X(\frac{y}{\alpha};nx)
-b y\sum_{n\in\mathbb{Z}} e^{-2\pi\alpha y n^2}\vartheta_X(\frac{y}{2 \alpha};nx)\Big).
\endaligned\end{equation}
\end{lemma}

By Lemma \ref{3Lemma2}, we have:

\begin{lemma}\label{3Lemma3} We have the following identity for the partial $x$-derivative of the function $(\mathcal{K}(\alpha;z)-b \mathcal{K}(2\alpha;z))$:
\begin{equation}\aligned\nonumber
-\frac{\partial}{\partial{x}}\big(\mathcal{K}(\alpha;z)-b \mathcal{K}(2\alpha;z)\big)=\frac{\sqrt{2}}{4\pi}\alpha^{-\frac{5}{2}}y^{\frac{1}{2}}e^{-\pi\alpha y}
\cdot \big(I(\alpha;z;b)+\mathcal{E}(\alpha;z;b)\big),
\endaligned\end{equation}
where
\begin{equation}\aligned\label{3eq2}
I(\alpha;z;b):=&-2\sqrt{2}\alpha \vartheta_{Y}(\frac{y}{\alpha};x)-4\sqrt{2}\pi \alpha^2y\vartheta_{Y}(\frac{y}{\alpha};x)-4\sqrt{2}y\vartheta_{XY}(\frac{y}{\alpha};x)\\
&+b \alpha  {e}^{-\pi\alpha y}\vartheta_{Y}(\frac{y}{2 \alpha};x)+4\pi b\alpha^2ye^{-\pi\alpha  y }\vartheta_{Y}(\frac{y}{2 \alpha};x)
+by e^{-\pi\alpha y}\vartheta_{XY}(\frac{y}{2 \alpha};x),\\
\endaligned\end{equation}
and
\begin{equation}\aligned\label{3eq3}
\mathcal{E}(\alpha;z;b):=&-2\sqrt{2}\alpha\underset{n=2}{\overset{\infty}{\sum}}n {e}^{-\pi\alpha y (n^2-1)}\vartheta_{Y}(\frac{y}{\alpha};nx)
+b \alpha \underset{n=2}{\overset{\infty}{\sum}} n {e}^{-\pi\alpha y (2 n^2-1)}\vartheta_{Y}(\frac{y}{2 \alpha};nx)\\
&-4\sqrt{2}\pi \alpha^2y\underset{n=2}{\overset{\infty}{\sum}} n^3e^{-\pi\alpha y (n^2-1)}\vartheta_{Y}(\frac{y}{\alpha};nx)
+4\pi b\alpha^2y\underset{n=2}{\overset{\infty}{\sum}} n^{3}e^{-\pi\alpha  y (2{n}^2-1)}\vartheta_{Y}(\frac{y}{2 \alpha};nx)\\
&-4\sqrt{2}y\underset{n=2}{\overset{\infty}{\sum}} n e^{-\pi\alpha y (n^2-1)}\vartheta_{XY}(\frac{y}{\alpha};nx)
+by\underset{n=2}{\overset{\infty}{\sum}}n e^{-\pi\alpha y(2 n^2-1)}\vartheta_{XY}(\frac{y}{2 \alpha};nx).
\endaligned\end{equation}
\end{lemma}

In view of Lemma \ref{3Lemma3}, Theorem \ref{Thm2} is equivalent to the following lemma:
\begin{lemma}\label{lem3.4} Assume that $ \alpha \geq \frac{3}{2},\: b\leq 2\sqrt{2}$. Then for $z\in \mathcal{D_{G}}$, it holds that
\begin{equation}\aligned\nonumber
I(\alpha;z;b)+\mathcal{E}(\alpha;z;b)>0.
\endaligned\end{equation}
\end{lemma}

 To explain Lemma \ref{lem3.4}, we aim to illustrate that $I(\alpha;z;b)$ acts as the major term, which is positive, while $\mathcal{E}(\alpha;z;b)$ represents an error term that does not change the sign of the whole expression. Based on \eqref{3eq2} and \eqref{3eq3}, we will provide a deformation of $I(\alpha;z;b)$ and an upper bound for $\mathcal{E}(\alpha;z;b)$ in the following lemma:

\begin{lemma}\label{3Lemma4} Assume that $\alpha>0,\:b\leq 2\sqrt{2}$. Then for $z\in \mathcal{D_{G}}$, it holds that
\begin{itemize}
  \item [(1)] Deformation of $I(\alpha;z;b):$
  \begin{equation}\aligned\nonumber
I(\alpha;z;b)=2\sqrt{2}\cdot\big(-\vartheta_{Y}(\frac{y}{\alpha};x)\big)\cdot\big(\alpha+2\pi{\alpha}^2y+2y\:\frac{\vartheta_{XY}(\frac{y}{\alpha};x)}{\vartheta_{Y}(\frac{y}{\alpha};x)}\big)\\
+b  {e}^{-\pi\alpha y}\vartheta_{Y}(\frac{y}{2 \alpha};x)\cdot\big( \alpha+4\pi{\alpha}^2 y+y\: \frac{\vartheta_{XY}(\frac{y}{2\alpha};x)}{\vartheta_{Y}(\frac{y}{2\alpha};x)}\big);
\endaligned\end{equation}
    \item [(2)] Upper bound for $\left|\mathcal{E}(\alpha;z;b)\right|:$
    \begin{equation}\aligned\nonumber
&\left|\mathcal{E}(\alpha;z;b)\right|\leq2\sqrt{2}\big(-\vartheta_{Y}(\frac{y}{\alpha};x)\big)\Big(\alpha\underset{n=2}{\overset{\infty}{\sum}}n {e}^{-\pi\alpha y (n^2-1)}\left|\frac{\vartheta_{Y}(\frac{y}{\alpha};nx)}{\vartheta_{Y}(\frac{y}{\alpha};x)}\right| \\
      &\:\:\:\:\:\:\:\:+2\pi \alpha^2y\underset{n=2}{\overset{\infty}{\sum}} n^3e^{-\pi\alpha y (n^2-1)}\left|\frac{\vartheta_{Y}(\frac{y}{\alpha};nx)}{\vartheta_{Y}(\frac{y}{\alpha};x)}\right|+2y\underset{n=2}{\overset{\infty}{\sum}} n e^{-\pi\alpha y (n^2-1)}\left|\frac{\vartheta_{XY}(\frac{y}{\alpha};nx)}{\vartheta_{Y}(\frac{y}{\alpha};x)}\right|\Big)\\
      &\:\:\:\:\:\:\:\:+2\sqrt{2}\big(-\vartheta_{Y}(\frac{y}{2 \alpha};x)\big)\Big( \alpha \underset{n=2}{\overset{\infty}{\sum}} n {e}^{-\pi\alpha y (2 n^2-1)}\left|\frac{\vartheta_{Y}(\frac{y}{2 \alpha};nx)}{\vartheta_{Y}(\frac{y}{2 \alpha};x)}\right|\\
      &\:\:\:\:\:\:\:\:+4\pi \alpha^2y\underset{n=2}{\overset{\infty}{\sum}} n^{3}e^{-\pi\alpha  y (2{n}^2-1)}\left|\frac{\vartheta_{Y}(\frac{y}{2 \alpha};nx)}{\vartheta_{Y}(\frac{y}{2 \alpha};x)}\right|
+y\underset{n=2}{\overset{\infty}{\sum}}n e^{-\pi\alpha y(2 n^2-1)}\left|\frac{\vartheta_{XY}(\frac{y}{2 \alpha};nx)}{\vartheta_{Y}(\frac{y}{2 \alpha};x)}\right|\Big).
\endaligned\end{equation}
\end{itemize}
\end{lemma}
\begin{proof}
Item (1) follows by \eqref{3eq2}. Item (2) is derived from \eqref{3eq3}, Remark \ref{2rem1}, and $b\leq 2\sqrt{2}$.
\end{proof}

We will divide the proof of Lemma \ref{lem3.4} into three cases, where {\bf Case A}: $\frac{y}{\alpha}\geq\frac{1}{2}$, {\bf Case B}: $\frac{y}{\alpha}\leq\frac{1}{4}$ and {\bf Case C}: $\frac{y}{\alpha}\in[\frac{1}{4},\frac{1}{2}]$, which will be presented separately in the next three subsections.

\subsection{Case A of Lemma \ref{lem3.4}: $\frac{y}{\alpha}\geq\frac{1}{2}$} In this subsection, we shall prove that

\begin{lemma}\label{3Lemma9} Assume that $\alpha\geq \frac{5}{4},\:\frac{y}{\alpha}\geq\frac{1}{2}\:\:and\:\:b\leq 2\sqrt{2}$, then for $z\in \mathcal{D_{G}}$, it holds that
\begin{itemize}
  \item [(1)] The lower bound function of $I(\alpha;z;b):$
  \begin{equation}\aligned\nonumber
I(\alpha;z;b)\geq &8\sqrt{2}\pi e^{-\pi\frac{y}{\alpha}}\sin(2\pi x)\big(1-\mu(\frac{1}{2})\big)\big(\alpha+2\pi{\alpha}^2y-2\pi y\:\frac{1+\nu(\frac{1}{2})}{1+\mu(\frac{1}{2})}\big)\\
      &-8\sqrt{2}\pi e^{-\pi y(\alpha+\frac{1}{2\alpha})}\sin(2\pi x)\big(1+\mu(\frac{1}{4})\big)\big(\alpha+4\pi{\alpha}^2y-\pi y\:\frac{1+\hat{\nu}(\frac{1}{4})}{1+\hat{\mu}(\frac{1}{4})}\big).
\endaligned\end{equation}
    \item [(2)] The upper bound function of $\left|\mathcal{E}(\alpha;z;b)\right|:$
    $$\left|\mathcal{E}(\alpha;z;b)\right|\leq 56\sqrt{2}\pi e^{-\pi \frac{y}{\alpha}}\sin(2\pi x)\cdot 10^{-3}.$$
    \item [(3)] $I(\alpha;z;b)+ \mathcal{E}(\alpha;z;b)\geq \frac{2 \sqrt{2}}{25}\pi e^{-\pi \frac{y}{\alpha}}\sin(2\pi x)>0.$
        \end{itemize}
\end{lemma}
\begin{proof}
By Lemmas \ref{3Lemma4}, \ref{2lem4} and \ref{2lem4add}, one has
\begin{equation}\aligned\label{3.2eq1}
I(\alpha;z;b)\geq &8\sqrt{2}\pi e^{-\pi\frac{y}{\alpha}}\sin(2\pi x)\big(1-\mu(\frac{y}{\alpha})\big)\big(\alpha+2\pi{\alpha}^2y-2\pi y\:\frac{1+\nu(\frac{y}{\alpha})}{1+\mu(\frac{y}{\alpha})}\big)\\
 &-8\sqrt{2}\pi \: e^{-\pi y(\alpha+\frac{1}{2\alpha})}\sin(2\pi x)\big(1+\mu(\frac{y}{2\alpha})\big)\big(\alpha+4\pi{\alpha}^2y-\pi y\:\frac{1+\hat{\nu}(\frac{y}{2\alpha})}{1+\hat{\mu}(\frac{y}{2\alpha})}\big).
\endaligned\end{equation}
Notice that $\mu(X)$ and $\nu(X)$ are decreasing as $X\geq\frac{1}{4}$.
Then as $\frac{y}{\alpha}\geq\frac{1}{2}$, one has
\begin{equation}\aligned\label{3.2eq2}
\mu(\frac{y}{\alpha})\leq \mu(\frac{1}{2})=0.0359\cdots,\:\:\:\:\mu(\frac{y}{2\alpha})\leq \mu(\frac{1}{4})=0.3960\cdots.
\endaligned\end{equation}
Notice that for $X\geq\frac{1}{4}$, we have
\begin{equation}\aligned\label{3.2eq3}
\frac{1+\nu(X)}{1+\mu(X)},\;\; \frac{1+\mu(X)}{1-\mu(X)}, \;\;\frac{1+\nu(X)}{1-{\mu}(X)}\;\;\mathrm{\:\:are\: \:decreasing },\:\:\frac{1+\hat{\nu}(X)}{1+\hat{\mu}(X)} \mathrm{\:\:is\: \:increasing}.
\endaligned\end{equation}
Then, there holds that
\begin{equation}\aligned\label{3.2eq4}
-0.5759\cdots=\frac{1+\hat{\nu}(\frac{1}{4})}{1+\hat{\mu}(\frac{1}{4})}\leq \frac{1+\hat{\nu}(\frac{y}{2\alpha})}{1+\hat{\mu}(\frac{y}{2\alpha})},\;\;\;\;
\frac{1+\nu(\frac{y}{\alpha})}{1+\mu(\frac{y}{\alpha})}\leq\frac{1+\nu(\frac{1}{2})}{1+\mu(\frac{1}{2})}=1.1042\cdots.
\endaligned\end{equation}
Therefore, \eqref{3.2eq1}-\eqref{3.2eq4} yield item (1). By Lemmas \ref{3Lemma4} and \ref{2lem4}, one gets
\begin{equation}\aligned\label{3.2eq5}
&\left|\mathcal{E}(\alpha;z;b)\right|\leq 8\sqrt{2}\pi e^{-\pi \frac{y}{\alpha}}\sin(2\pi x)\Big(\big(1+\mu(\frac{y}{2\alpha})\big)\big( \alpha\:\underset{n=2}{\overset{\infty}{\sum}}n^2 e^{-\pi{y}(\alpha(2n^2-1)-\frac{1}{2\alpha})}\cdot\frac{1+\mu(\frac{y}{2\alpha})}{1-\mu(\frac{y}{2\alpha})}\\
&\;\;+4\pi{\alpha}^2 y \underset{n=2}{\overset{\infty}{\sum}}n^4 e^{-\pi{y}(\alpha(2n^2-1)-\frac{1}{2\alpha})}\frac{1+\mu(\frac{y}{2\alpha})}{1-\mu(\frac{y}{2\alpha})}+\pi y \underset{n=2}{\overset{\infty}{\sum}}n^2 e^{-\pi{y}(\alpha(2n^2-1)-\frac{1}{2\alpha})}\frac{1+\nu(\frac{y}{2\alpha})}{1-\mu(\frac{y}{2\alpha})}\big)\\
&\;\;+\big(1+\mu(\frac{y}{\alpha})\big)\cdot\big( \alpha\cdot  \mu(\alpha y)\cdot\frac{1+\mu(\frac{y}{\alpha})}{1-\mu(\frac{y}{\alpha})}
+2\pi{\alpha}^2 y \cdot \nu(\alpha y)\cdot \frac{1+\mu(\frac{y}{\alpha})}{1-\mu(\frac{y}{\alpha})}+2\pi y \cdot\mu(\alpha y)\cdot \frac{1+\nu(\frac{y}{\alpha})}{1-\mu(\frac{y}{\alpha})}\big) \Big).
\endaligned\end{equation}
Note that $z\in \mathcal{D_{G}}$ implies $y>\frac{\sqrt{3}}{2}$. Then for $\alpha\geq \frac{5}{4}$ and $\frac{y}{\alpha}\geq\frac{1}{2}$, by \eqref{3.2eq3}, \eqref{3.2eq5}, along with the monotonicity of $\mu(X)$ and $\nu(X)$, item (2) is deduced. Items (1) and (2) yield item (3).
\end{proof}

\subsection{Case B of Lemma \ref{lem3.4}: $\frac{y}{\alpha}\leq\frac{1}{4}$} In this subsection, we shall prove that

\begin{lemma}\label{3Lemma14} Suppose $\alpha\geq 1,\:\frac{y}{\alpha}\in (0,\frac{1}{4}]\:\:and\:\:b\leq 2\sqrt{2}$, then for $z\in \mathcal{D_{G}}$, it holds that
\begin{itemize}
  \item [(1)] The lower bound function of $I(\alpha;z;b):$
  \begin{equation}\aligned\nonumber
I(\alpha;z;b)\geq &\sin(2\pi x)\big(\frac{\alpha}{y}\big)^{\frac{3}{2}}\big(2\sqrt{2}\pi e^{-\frac{\pi\alpha}{4y}}(\alpha+2\pi{\alpha}^2 y- y\cdot\frac{3(\frac{y}{\alpha})^2+8{\pi}^2 e^{-4\pi}}{(\frac{y}{\alpha})^3-\frac{\pi}{4} e^{-4\pi}}   )\\
    &\;\;-8 e^{-\pi\alpha y}(\alpha+4\pi{\alpha}^2 y+\frac{\pi{\alpha}^2}{y})\big).
\endaligned\end{equation}
    \item [(2)] The upper bound function of $\left|\mathcal{E}(\alpha;z;b)\right|:$
    $$\left|\mathcal{E}(\alpha;z;b)\right|\leq \frac{\sqrt{2}\pi }{25}e^{-\frac{\pi\alpha}{4y}}\sin(2\pi x)\big(\frac{\alpha}{y}\big)^{\frac{3}{2}}.$$
    \item [(3)] $I(\alpha;z;b)+ \mathcal{E}(\alpha;z;b)\geq 2\sqrt{2}\pi e^{-\frac{\pi\alpha}{4y}}\sin(2\pi x)\big(\frac{\alpha}{y}\big)^{\frac{3}{2}}>0.$
        \end{itemize}
\end{lemma}
\begin{proof}
(1). By Lemmas \ref{3Lemma4}, \ref{2lem4} and \ref{2lem4add}, one has
\begin{equation}\aligned\nonumber
I(\alpha;z;b)\geq &\sin(2\pi x)\big(\frac{\alpha}{y}\big)^{\frac{3}{2}}\Big(2\sqrt{2}\pi e^{-\frac{\pi\alpha}{4y}}\big(\alpha+2\pi{\alpha}^2 y- y\:\frac{3(\frac{y}{\alpha})^2+8{\pi}^2 e^{-\frac{\pi\alpha}{y}}}{(\frac{y}{\alpha})^3-4\pi(\frac{y}{\alpha})^2 e^{-\frac{\pi\alpha}{y}}}  \big)\\
&\;\;\;\;\;\;\;-8 e^{-\pi\alpha y}\big(\alpha+4\pi{\alpha}^2 y+\frac{\pi{\alpha}^2}{y}\big)\Big).
\endaligned\end{equation}
 Note that $\frac{\alpha}{y}\geq4$, then $e^{-\frac{\pi\alpha}{y}}\leq e^{-4\pi}$ and $(\frac{y}{\alpha})^2 e^{-\frac{\pi\alpha}{y}}\leq \frac{1}{16}e^{-4\pi}$. Thus, (1) is obtained.

 (2). Notice that $z\in \mathcal{D_{G}}$ implies $y>\frac{\sqrt{3}}{2}$. By Lemmas \ref{3Lemma4} and \ref{2lem4}-\ref{2lem4add}, one has
\begin{equation}\aligned\label{3lemma16eq}
\left|\mathcal{E}(\alpha;z;b)\right|\leq &2\sqrt{2}\pi e^{-\frac{\pi\alpha}{4y}}\sin(2\pi x)\big(\frac{\alpha}{y}\big)^{\frac{3}{2}}\big(
\frac{\alpha}{{\pi}^2} \underset{n=2}{\overset{\infty}{\sum}}n^2 e^{-\pi{\alpha}[y(n^2-1)-\frac{1}{2y}]}+\frac{2}{\pi}{\alpha}^2y \underset{n=2}{\overset{\infty}{\sum}}n^4 e^{-\pi{\alpha}[y(n^2-1)-\frac{1}{2y}]}\\
&+\frac{{\alpha}^2}{2\pi y}\underset{n=2}{\overset{\infty}{\sum}}n^2 e^{-\pi{\alpha}[y(n^2-1)-\frac{1}{2y}]}
+\frac{2\sqrt{2}}{{\pi}^2}\alpha\:\underset{n=2}{\overset{\infty}{\sum}}n^2 e^{-\pi{\alpha}[y(2n^2-1)-\frac{3}{4y}]}\\
&+\frac{8\sqrt{2}}{\pi}{\alpha}^2y\underset{n=2}{\overset{\infty}{\sum}}n^4 e^{-\pi{\alpha}[y(2n^2-1)-\frac{3}{4y}]}+
\frac{2\sqrt{2}}{\pi}\alpha^2 y^{-1}\underset{n=2}{\overset{\infty}{\sum}}n^2 e^{-\pi{\alpha}[y(2n^2-1)-\frac{3}{4y}]}
\big).\\
\endaligned\end{equation}
Given that the terms in \eqref{3lemma16eq} are exponentially decaying, they can be effectively bounded.

(3). Item (3) follows by items (1) and (2).
\end{proof}

\subsection{Case C of Lemma \ref{lem3.4}: $\frac{1}{4}\leq\frac{y}{\alpha}\leq \frac{1}{2}$} In this subsection, we shall prove that

\begin{lemma}\label{3Lemma19} Assume that $\alpha\geq \frac{3}{2},\:\frac{1}{4}\leq\frac{y}{\alpha}\leq \frac{1}{2}$, and $b\leq 2\sqrt{2}$. Then for $z\in \mathcal{D_{G}}$, it holds that
\begin{itemize}
  \item [(1)] Lower bound of $I(\alpha;z;b):$
  \begin{equation}\aligned\nonumber
I(\alpha;z;b)\geq & 8e^{-\pi\frac{y}{\alpha}}\sin(2\pi x)\Big(\sqrt{2}\pi \big(1-\mu(\frac{1}{4})\big)\big(\alpha+2\pi{\alpha}^2 y-2\pi y\cdot\frac{1+\nu(\frac{1}{4})}{1+\mu(\frac{1}{4})}\big)\\
    &-e^{\frac{\pi}{2}}\cdot e^{-\pi \alpha y} {\alpha}^{\frac{5}{2}}y^{-\frac{3}{2}}(1+4\pi\alpha y+\pi\alpha {y}^{-1})\Big).
\endaligned\end{equation}
    \item [(2)] Upper bound of $\left|\mathcal{E}(\alpha;z;b)\right|:$
    $$\left|\mathcal{E}(\alpha;z;b)\right|\leq \frac{6 \sqrt{2}\pi}{125} e^{-\pi \frac{y}{\alpha}}\sin(2\pi x).$$
    \item [(3)] $I(\alpha;z;b)+ \mathcal{E}(\alpha;z;b)\geq \frac{2\sqrt{2}\pi}{25} e^{-\pi\frac{y}{\alpha}}\sin(2\pi x)>0.$
        \end{itemize}
\end{lemma}
\begin{proof}
(1). For $I(\alpha;z;b)$, as $b\leq 2\sqrt{2}$, by Lemmas \ref{3Lemma4}, \ref{2lem4}, \ref{2lem4add}, one has
\begin{equation}\aligned\nonumber
I(\alpha;z;b)\geq & 8e^{-\pi\frac{y}{\alpha}}\sin(2\pi x)\Big(\sqrt{2}\pi \big(1-\mu(\frac{1}{4})\big)\big(\alpha+2\pi{\alpha}^2 y-2\pi y\cdot\frac{1+\nu(\frac{1}{4})}{1+\mu(\frac{1}{4})}\big)\\
&\;\;-e^{\frac{\pi y}{\alpha}} \cdot  e^{-\pi \alpha y} \cdot{\alpha}^{\frac{5}{2}}y^{-\frac{3}{2}}(1+4\pi\alpha y+\pi\alpha {y}^{-1})\Big).
\endaligned\end{equation}
Note that $\frac{1}{4}\leq\frac{y}{\alpha}\leq\frac{1}{2}$, then one has $e^{\frac{\pi y}{\alpha}}\leq e^{\frac{\pi }{2}}.$ Thus, item (1) is obtained.

(2). The estimate for (2) here is similar to that for (2) in Lemma \ref{3Lemma9}, hence, we omit its details.

(3). $1-\mu(\frac{1}{4})=0.6039\cdots,  \:\:\:\: \frac{1+\nu(\frac{1}{4})}{1+\mu(\frac{1}{4})}=1.9123\cdots.$
Item (3) is deduced by items (1) and (2).
\end{proof}

\section{Minimization on the vertical line $\Gamma_{c}$}

Recalling Theorem \ref{3Thm1}, as $ \alpha \geq \frac{3}{2},\:b\leq 2\sqrt{2}$, we have
\begin{equation}\aligned\nonumber
\underset{z\in \mathbb{H}}{\min}\big(\mathcal{K}(\alpha;z)-b \mathcal{K}(2\alpha;z)\big)
=\underset{z\in \Gamma_{c}}{\min}\big(\mathcal{K}(\alpha;z)-b \mathcal{K}(2\alpha;z)\big),\\
\endaligned\end{equation}
In this section, we aim to establish the following result:

\begin{theorem}\label{4Thm1}  Assume that $ \alpha \geq 2$. Then, for $b\leq 2$, up to the action by the modular group,
\begin{equation}\aligned\nonumber
\argmin_{z\in\Gamma_{c}}\big(\mathcal{K}(\alpha;z)-b \mathcal{K}(2\alpha;z)\big)=e^{i\frac{\pi}{3}}.
\endaligned\end{equation}
\end{theorem}
Regarding to $\mathcal{K}(\alpha;z)$, it is known that
\begin{proposition}[Luo-Wei \cite{Luo2023}]\label{Pro}For $\alpha\geq2$, up to the modular group, then
\begin{equation}\aligned\nonumber
\argmin_{z\in\Gamma_{c}}\mathcal{K}(\alpha;z)=e^{i\frac{\pi}{3}}.
\endaligned\end{equation}
\end{proposition}

  By Proposition \ref{Pro} and the following deformation:
  $$\mathcal{K}(\alpha;z)-b \mathcal{K}(2\alpha;z)=\mathcal{K}(\alpha;z)-2 \mathcal{K}(2\alpha;z)+(2-b)\mathcal{K}(2\alpha;z),$$
  to prove Theorem \ref{4Thm1}, it suffices to prove that
  \begin{proposition}\label{PropK2}  Assume that $ \alpha \geq 2$. Then, up to the action by the modular group,
\begin{equation}\aligned\nonumber
\argmin_{z\in\Gamma_{c}}\big(\mathcal{K}(\alpha;z)-2\mathcal{K}(2\alpha;z)\big)=e^{i\frac{\pi}{3}}.
\endaligned\end{equation}
\end{proposition}
This proof of Proposition \ref{PropK2} is based on Propositions \ref{4prop1} and \ref{4prop2} (an illustration the proof can be found in Figure \ref{a}).
\begin{figure}
\centering
 \includegraphics[scale=0.4]{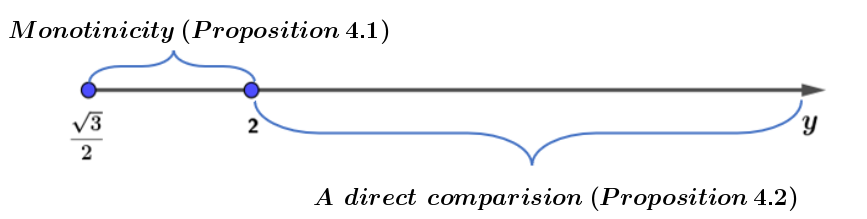}
 \caption{A diagram of the proof of Theorem \ref{4Thm1}.}
 \label{a}
\end{figure}

\begin{proposition}\label{4prop1}Suppose that $\alpha\geq 2$. Then for $y\in[\frac{\sqrt{3}}{2},2]$, it holds that
\begin{equation}\aligned\nonumber
\frac{\partial}{\partial{y}}\big(\mathcal{K}(\alpha;\frac{1}{2}+iy)-2 \mathcal{K}(2\alpha;\frac{1}{2}+iy)\big)\geq 0.
\endaligned\end{equation}
\end{proposition}

For $y\geq2$, we use a direct method as follows:
\begin{proposition}[A direct comparison]\label{4prop2}Suppose that $\alpha\geq 2$. Then for $y\geq 2$, it holds
\begin{equation}\aligned\nonumber
\mathcal{K}(\alpha;\frac{1}{2}+iy)-2 \mathcal{K}(2\alpha;\frac{1}{2}+iy)> \mathcal{K}(\alpha;\frac{1}{2}+i\frac{\sqrt{3}}{2})-2 \mathcal{K}(2\alpha;\frac{1}{2}+i\frac{\sqrt{3}}{2}).
\endaligned\end{equation}
\end{proposition}

We provide the proof of Propositions \ref{4prop1} and \ref{4prop2} in Subsections 4.1 and 4.2, respectively.
\subsection{Proof of Proposition \ref{4prop1}}
 By Proposition 3.4 of B\'etermin \cite{Bet2018}:
 $$y^2\frac{\partial}{\partial{y}}\big(\mathcal{K}(\alpha;\frac{1}{2}+iy)-2 \mathcal{K}(2\alpha;\frac{1}{2}+iy)\big)\big|_{y=\frac{\sqrt{3}}{2}}=0,\;\;\mathrm{for\;}\alpha\geq2,$$
  and
  $\frac{1}{y^2}\frac{\partial}{\partial{y}}(y^2\frac{\partial}{\partial{y}})=\frac{\partial^2}{\partial{y}^2}+\frac{2}{y}\frac{\partial}{\partial{y}},$
 to prove Proposition \ref{4prop1}, it suffices to prove Lemma \ref{4lemma1}.
In this subsection, we aim to prove that
\begin{lemma}\label{4lemma1}Assume that $ \alpha \geq 2$. Then for $y\in[\frac{\sqrt{3}}{2},2]$,
\begin{equation}\aligned\nonumber
\big(\frac{\partial^2}{\partial{y}^2}+\frac{2}{y}\frac{\partial}{\partial{y}}\big)\big(\mathcal{K}(\alpha;\frac{1}{2}+iy)-2 \mathcal{K}(2\alpha;\frac{1}{2}+iy)\big)>0.
\endaligned\end{equation}
\end{lemma}

To better illustrate the proof of Lemma \ref{4lemma1}, we denote that
\begin{equation}\aligned\label{4eq1}
\mathcal{S}_{1}(\alpha;y)&:= \underset{n,m}{\sum}n^2e^{-\pi\alpha(yn^2+\frac{(m+\frac{n}{2})^2}{y})},\\
\mathcal{S}_{2}(\alpha;y)&:=\underset{n,m}{\sum}(n^2-\frac{(m+\frac{n}{2})^2}{y^2})^2e^{-\pi\alpha(yn^2+\frac{(m+\frac{n}{2})^2}{y})},\\
\mathcal{S}_{3}(\alpha;y)&:=\underset{n,m}{\sum}n^2(yn^2+\frac{(m+\frac{n}{2})^2}{y})
e^{-\pi\alpha(yn^2+\frac{(m+\frac{n}{2})^2}{y})},\\
\mathcal{S}_{4}(\alpha;y)&:=\underset{n,m}{\sum}(n^2-\frac{(m+\frac{n}{2})^2}{y^2})^2(yn^2+\frac{(m+\frac{n}{2})^2}{y})
e^{-\pi\alpha(yn^2+\frac{(m+\frac{n}{2})^2}{y})}.\\
\endaligned\end{equation}

 Using the expression of $\mathcal{K}(\alpha;z)$ given by \eqref{define}, we give the expression of $(\frac{\partial^2}{\partial{y}^2}+\frac{2}{y}\frac{\partial}{\partial{y}})\big(\mathcal{K}(\alpha;\frac{1}{2}+iy)-2 \mathcal{K}(2\alpha;\frac{1}{2}+iy)\big)$ in the following lemma. The proof of this lemma is straightforward and thus omitted.
\begin{lemma}[The expression of $(\frac{\partial^2}{\partial{y}^2}+\frac{2}{y}\frac{\partial}{\partial{y}})(\mathcal{K}(\alpha;\frac{1}{2}+iy)-2 \mathcal{K}(2\alpha;\frac{1}{2}+iy))$]\label{4lemma2}
\begin{equation}\aligned\nonumber
&\big(\frac{\partial^2}{\partial{y}^2}+\frac{2}{y}\frac{\partial}{\partial{y}}\big)\big(\mathcal{K}(\alpha;\frac{1}{2}+iy)-2 \mathcal{K}(2\alpha;\frac{1}{2}+iy)\big)\\
=&\:\:\frac{2}{y}\mathcal{S}_{1}(\alpha;y)+8\pi\alpha\mathcal{S}_{2}(2\alpha;y)+\frac{8\pi\alpha}{y}\mathcal{S}_{3}(2\alpha;y)+{\pi}^2{\alpha}^2\mathcal{S}_{4}(\alpha;y)\\
&-\frac{4}{y}\mathcal{S}_{1}(2\alpha;y)-2\pi\alpha\mathcal{S}_{2}(\alpha;y)-\frac{2\pi\alpha}{y}\mathcal{S}_{3}(\alpha;y)-8{\pi}^2{\alpha}^2\mathcal{S}_{4}(2\alpha;y).
\endaligned\end{equation}
\end{lemma}

In view of the expression provided in Lemma \ref{4lemma2}, we first provide the lower bound estimates of $\mathcal{S}_{i}(\alpha;y)$ $(i=1,\cdots,4)$.

\begin{lemma}[The lower bound estimates]\label{lem4.4} For $\alpha,y>0$, it holds that
\begin{itemize}
        \item [(2)] $\mathcal{S}_{1}(\alpha;y)\geq 4e^{-\pi\alpha(y+\frac{1}{4y})};$
       \item [(3)] $\mathcal{S}_{2}(2\alpha;y)\geq \frac{2}{y^4}e^{-2\pi\frac{\alpha}{y}}+4(1-\frac{1}{4y^2})^2e^{-2\pi\alpha(y+\frac{1}{4y})};$
    \item [(4)] $\mathcal{S}_{3}(2\alpha;y)\geq  (4y+\frac{1}{y})e^{-2\pi\alpha(y+\frac{1}{4y})};$
    \item [(1)] $\mathcal{S}_{4}(\alpha;y)\geq \frac{2}{y^5}e^{-\frac{\pi\alpha}{y}}+4(1-\frac{1}{4y^2})^2(y+\frac{1}{4y})e^{-\pi\alpha(y+\frac{1}{4y})}.$
        \end{itemize}
\end{lemma}

The following results were derived in Luo-Wei (\cite{LW2022}).

\begin{lemma}[Luo-Wei \cite{LW2022}]\label{4lemma8} For $\alpha\geq2,y\in[\frac{\sqrt{3}}{2},2]$, it holds that
\begin{itemize}
  \item [(1)] $\mathcal{S}_{1}(2\alpha;y)\leq4e^{-2\pi\alpha(y+\frac{1}{4y})}\cdot(1+\epsilon_{a});$
  \item [(2)] $\mathcal{S}_{2}(\alpha;y)\leq\frac{2}{y^4}e^{-\pi\frac{\alpha}{y}}\cdot(1+\epsilon_{b}),$
\end{itemize}
where
\begin{equation}\aligned\label{4lem7eq1}\nonumber
{\epsilon}_{a}:=&2e^{-2\pi\alpha(3y-\frac{1}{4y})}(1+\underset{n=2}{\overset{\infty}{\sum}}n^2e^{-8\pi\alpha y(n^2-1)})\cdot(1+2\underset{n=1}{\overset{\infty}{\sum}}e^{-\pi{n}^2\frac{2\alpha}{y}})+\underset{n=2}{\overset{\infty}{\sum}}e^{-\frac{2\pi\alpha}{y}n\cdot(n-1)}\\
&+\underset{n=2}{\overset{\infty}{\sum}}(2n-1)^{2}e^{-8{\pi}{\alpha}y(n-1)\cdot n}
+\underset{n=2}{\overset{\infty}{\sum}}(2n-1)^{2}e^{-8{\pi}{\alpha}y(n-1)\cdot n}\underset{m=2}{\overset{\infty}{\sum}}e^{-\frac{2\pi\alpha}{y}m\cdot(m-1)},\\
\endaligned\end{equation}
and $\epsilon_{b}:=\epsilon_{b,1}+\epsilon_{b,2}+\epsilon_{b,3}+\epsilon_{b,4}$ and each $\epsilon_{b,i}\:(i=1,2,3,4)$ is expressed as follows
\begin{equation}\aligned\nonumber
\epsilon_{b,1}
:&=2y^4 e^{-\pi\alpha y}\cdot(1+\sum_{n=2}^\infty e^{-\frac{4\pi\alpha}{y}n\cdot(n-1)})\cdot
(1+\sum_{n=2}^\infty (2n-1)^4 e^{-4\pi\alpha y n(n-1)}),\\
\epsilon_{b,2}
:&=\frac{1}{8}e^{-\pi\alpha y}\cdot(1+\sum_{n=2}^\infty (2n-1)^4 e^{-\frac{4\pi\alpha}{y}n(n-1)})
\cdot(1+\sum_{n=2}^\infty e^{-4\pi\alpha yn(n-1)}),\\
\epsilon_{b,3}
:&=16y^4 e^{-\pi \alpha (4y-\frac{1}{y})}\cdot(1+\sum_{n=2}^\infty n^4e^{-4\pi\alpha y(n^2-1)})\cdot(1+2\sum_{n=1}^\infty e^{-\pi\frac{\alpha}{y}n^2}),\\
\epsilon_{b,4}
:&=y^4 e^{-\pi \alpha (4y-\frac{1}{y})}\cdot(1+\sum_{n=2}^\infty e^{-4\pi\alpha y(n^2-1)})\cdot(1+2\sum_{n=1}^\infty \frac{n^4}{y^4}e^{-\pi\frac{\alpha}{y}n^2}).
\endaligned\end{equation}
Numerically, ${\epsilon}_{a}\leq4\cdot 10^{-6},$ and $\epsilon_{b}\leq6\cdot 10^{-3}.$
\end{lemma}

Next, we provide upper bound of $\mathcal{S}_{j}(\alpha;y), j=3,4$.
\begin{lemma}[Upper bound of $\mathcal{S}_{j}(\alpha;y), j=3,4$]\label{4lemma10} For $\alpha\geq2,y\in[\frac{\sqrt{3}}{2},2]$, it holds that
\begin{itemize}
  \item [(1)] $\mathcal{S}_{3}(\alpha;y)\leq4ye^{-\pi\alpha(y+\frac{1}{4y})}\cdot(1+\epsilon_{c}),$
  \item [(2)] $\mathcal{S}_{4}(2\alpha;y)
\leq \frac{2}{y^5}e^{-\frac{2\pi\alpha}{y}}\cdot(1+\epsilon_{d})+4ye^{-2\pi\alpha(y+\frac{1}{4y})}\cdot(1-\frac{1}{4y^2}-\frac{1}{16y^4}+\epsilon_{e}).$
\end{itemize}
where ${\epsilon}_{c}:=\underset{j=1}{\overset{6}{\sum}}\epsilon_{c,j}$.
Here each $\epsilon_{c,j}\:(j=1,2,\cdots,6)$ is small and expressed by
\begin{equation}\aligned\nonumber
\epsilon_{c,1}:&=\underset{n=2}{\overset{\infty}{\sum}}(2n-1)^4e^{-4{\pi}{\alpha}y(n-1)\cdot n},\:\:\:\:
\epsilon_{c,2}:=\underset{n=2}{\overset{\infty}{\sum}}e^{-\pi\alpha\frac{(n-1)\cdot n}{y}},\:\:\:\:\:
\epsilon_{c,3}:=\epsilon_{c,1}\cdot\epsilon_{c,2},\\
\epsilon_{c,4}:&=\frac{1}{4{y}^2}(1+\underset{n=2}{\overset{\infty}{\sum}}(2n-1)^2e^{-4{\pi}{\alpha}y(n-1)\cdot n})\cdot(1+\underset{n=2}{\overset{\infty}{\sum}}(2n-1)^2e^{-\pi\alpha\frac{(n-1)\cdot n}{y}}),\\
\epsilon_{c,5}:&=8 e^{-{\pi}{\alpha}(3y-\frac{1}{4y})}(1+\underset{n=2}{\overset{\infty}{\sum}}n^4e^{-4{\pi}{\alpha}y(n^2-1)})\cdot(1+2\underset{m=1}{\overset{\infty}{\sum}}e^{-\pi\alpha\frac{m^2}{y}}),\\
\epsilon_{c,6}:&=\frac{4}{{y}^2}e^{-3{\pi}{\alpha}(y+\frac{1}{4y})}(1+\underset{n=2}{\overset{\infty}{\sum}}n^2e^{-4\pi\alpha{y}(n^2-1)})\cdot(1+\underset{m=2}{\overset{\infty}{\sum}}
m^2e^{-\pi\alpha\frac{m^2-1}{y}}).\\
\endaligned\end{equation}
and
\begin{equation}\aligned\nonumber
{\epsilon}_{d}:=&2\underset{n=1}{\overset{\infty}{\sum}}e^{-8{\pi}{\alpha}y{n}^2}+\underset{n=2}{\overset{\infty}{\sum}}n^6e^{-2\pi\alpha\frac{n^2-1}{y}}+
2\underset{n=1}{\overset{\infty}{\sum}}e^{-8{\pi}{\alpha}y{n}^2} \underset{m=2}{\overset{\infty}{\sum}}m^6e^{-2\pi\alpha\frac{m^2-1}{y}}\;\;\;\;\;\;\;\;\;\;\\
&+64{y}^6e^{-2{\pi}{\alpha}(4y-\frac{1}{y})}(1+\underset{n=2}{\overset{\infty}{\sum}}n^{6}e^{-8{\pi}{\alpha}y(n^2-1)})\cdot(1+2\underset{n=1}{\overset{\infty}{\sum}}e^{-\pi{n}^2\frac{2\alpha}{y}}), \endaligned\end{equation}
\begin{equation}\aligned\nonumber
{\epsilon}_{e}:=&\:\:\frac{1}{64{y}^6}(1+\underset{n=2}{\overset{\infty}{\sum}}e^{-8{\pi}{\alpha}y(n-1)\cdot n})(1+\underset{n=2}{\overset{\infty}{\sum}}(2n-1)^{6}e^{-{2\pi}{\alpha}\frac{(n-1)\cdot n}{y}})+\underset{n=2}{\overset{\infty}{\sum}}e^{-2{\pi}{\alpha}\frac{(n-1)\cdot n}{y}}\\
&+\underset{n=2}{\overset{\infty}{\sum}}(2n-1)^{6}e^{-8{\pi}{\alpha}y(n-1)\cdot n}+\underset{n=2}{\overset{\infty}{\sum}}(2n-1)^{6}e^{-8{\pi}{\alpha}y(n-1)\cdot n}\underset{m=2}{\overset{\infty}{\sum}}e^{-2{\pi}{\alpha}\frac{(m-1)\cdot m}{y}}.
\endaligned\end{equation}
Numerically, $\epsilon_{c}\leq\frac{2}{5}$, $\epsilon_{d}\leq 5\cdot10^{-7},$ and $\epsilon_{e}\leq 4\cdot 10^{-2}.$
\end{lemma}
\begin{proof}
Since the proofs of the two items in Lemma \ref{4lemma10} are similar, we provide the proof for item (1) only to avoid repetition. For simplicity, we denote:
\begin{equation}\aligned\label{qiao}
\mathcal{S}_{3,a}(\alpha;y)&:=\underset{p\equiv q\equiv0(mod2) }{\sum}  y {p}^4 e^{-\pi\alpha(y{p}^2+\frac{q^2}{4y})},\:\:\:\:
\mathcal{S}_{3,b}(\alpha;y):=\underset{p\equiv q\equiv1(mod2) }{\sum}   y{p}^4 e^{-\pi\alpha(y{p}^2+\frac{q^2}{4y})},\\
\mathcal{S}_{3,c}(\alpha;y)&:=\underset{p\equiv q\equiv0(mod2) }{\sum}  \frac{p^2{q}^2}{4y} e^{-\pi\alpha(y{p}^2+\frac{q^2}{4y})},\:\:
\mathcal{S}_{3,d}(\alpha;y):=\underset{p\equiv q\equiv1(mod2) }{\sum}  \frac{p^2{q}^2}{4y} e^{-\pi\alpha(y{p}^2+\frac{q^2}{4y})}.
\endaligned\end{equation}
Then, by \eqref{4eq1}, the sum $\mathcal{S}_{3}(\alpha;y)$ can be divided into the following four parts:
\begin{equation}\aligned\label{moon}
\mathcal{S}_{3}(\alpha;y)=\underset{p\equiv q(mod2) }{\sum} p^2(yp^2+\frac{q^2}{4y}) e^{-\pi\alpha(y{p}^2+\frac{q^2}{4y})}=\mathcal{S}_{3,a}(\alpha;y)+\mathcal{S}_{3,b}(\alpha;y)+\mathcal{S}_{3,c}(\alpha;y)+\mathcal{S}_{3,d}(\alpha;y).
\endaligned\end{equation}
Next, we shall calculate the four parts respectively. Using \eqref{qiao}, one has
\begin{equation}\aligned\label{moon1}
\mathcal{S}_{3,a}(\alpha;y)&=\underset{p=2n,q=2m }{\sum}  y {p}^4 e^{-\pi\alpha(y{p}^2+\frac{q^2}{4y})}
=16y\:\underset{n }{\sum} n^4e^{-4\pi{\alpha}y{n}^2}\underset{m }{\sum}e^{-\pi\alpha\frac{m^2}{y}}\\
&=32y\:e^{-4\pi\alpha{y}}\cdot\big(1+\underset{n=2}{\overset{\infty}{\sum}}n^{4}e^{-4{\pi}{\alpha}y(n^2-1)}  \big)\cdot(1+2\underset{m=1}{\overset{\infty}{\sum}}e^{-\pi\alpha\frac{m^2}{y}})=\:\:4 ye^{-\pi\alpha(y+\frac{1}{4y})}\cdot\epsilon_{c,5},
\endaligned\end{equation}
and
\begin{equation}\aligned\label{moon2}
\mathcal{S}_{3,b}(\alpha;y)&=\underset{p=2n-1,q=2m-1 }{\sum}  y {p}^4 e^{-\pi\alpha(y{p}^2+\frac{q^2}{4y})}=4y\:\underset{n=1 }{\overset{\infty}{\sum}}(2n-1)^4e^{-\pi{\alpha}y(2n-1)^2}\underset{m=1 }{\overset{\infty}{\sum}}e^{-\pi{\alpha}\frac{(2m-1)^2}{4y}}\\
&=4y\:e^{-\pi\alpha(y+\frac{1}{4y})}\big(1+\underset{n=2}{\overset{\infty}{\sum}}(2n-1)^4e^{-4{\pi}{\alpha}y(n-1)\cdot n} \big)\cdot\big(1+\underset{m=2}{\overset{\infty}{\sum}}e^{-\pi\alpha\frac{(m-1)\cdot m}{y}} \big)\\
&=4y\:e^{-\pi\alpha(y+\frac{1}{4y})}\cdot(1+\epsilon_{c,1})\cdot(1+\epsilon_{c,2})=4ye^{-\pi\alpha(y+\frac{1}{4y})}
\cdot(1+\epsilon_{c,1}+\epsilon_{c,2}+\epsilon_{c,3}).
\endaligned\end{equation}
Similarly,
\begin{equation}\aligned\label{moon3}
\mathcal{S}_{3,c}(\alpha;y)&=\underset{p=2n,q=2m }{\sum} \frac{p^{2}q^{2}}{4y} e^{-\pi\alpha(y{p}^2+\frac{q^2}{4y})}=\frac{16}{y}\underset{n=1 }{\overset{\infty}{\sum}}n^2e^{-4\pi{\alpha}y{n}^2}\underset{m=1 }{\overset{\infty}{\sum}}m^2e^{-\pi{\alpha}\frac{m^2}{y}}\\
&=\frac{16}{y}e^{-4\pi\alpha(y+\frac{1}{4y})}\big(1+\underset{n=2}{\overset{\infty}{\sum}}n^2e^{-4\pi\alpha{y}(n^2-1)}\big)\big(1+\underset{m=2}{\overset{\infty}{\sum}}
m^2e^{-\pi\alpha\frac{m^2-1}{y}}\big)\:\:\:\:\:\:\:\:\:\:\:\:\:\:\:\:\:\:\:\:\:\:\:\:\:\:\:\:\:\:\:\:\:\:\:\:\:\:\:\:\:\:\:\:\:\:\:\:\:\:\:\:\:\:\:\:\:\:\:\:\:\:\:\:\:\:\:\\
&=4ye^{-\pi\alpha(y+\frac{1}{4y})}\cdot\epsilon_{c,6},
\endaligned\end{equation}
and
\begin{equation}\aligned\label{moon4}
\mathcal{S}_{3,d}(\alpha;y)&=\underset{p=2n-1,q=2m-1 }{\sum} \frac{p^{2}q^{2}}{4y} e^{-\pi\alpha(y{p}^2+\frac{q^2}{4y})}\\
&=\frac{1}{y}\underset{n=1 }{\overset{\infty}{\sum}}(2n-1)^2e^{-\pi{\alpha}y(2n-1)^2}\underset{m=1 }{\overset{\infty}{\sum}}(2m-1)^{2}e^{-\pi{\alpha}\frac{(2m-1)^2}{4y}}\\
&=\frac{1}{y}e^{-\pi\alpha(y+\frac{1}{4y})}(1+\underset{n=2}{\overset{\infty}{\sum}}(2n-1)^2e^{-4{\pi}{\alpha}y(n-1)\cdot n})\cdot(1+\underset{n=2}{\overset{\infty}{\sum}}(2n-1)^2e^{-\pi\alpha\frac{(n-1)\cdot n}{y}})\:\:\:\:\:\:\:\:\:\:\:\:\:\:\:\:\:\\
&=4 y\cdot e^{-\pi\alpha(y+\frac{1}{4y})}\cdot{\epsilon}_{c,4}.
\endaligned\end{equation}
Thus, \eqref{moon}-\eqref{moon4} together yield item (1).
\end{proof}

With Lemmas \ref{4lemma2}-\ref{4lemma10} established, we are now ready to prove Lemma \ref{4lemma1}.
 \begin{proof}  By Lemmas \ref{4lemma2}-\ref{4lemma10}, one has
\begin{equation}\aligned\label{4.1eq1}
\big(\frac{\partial^2}{\partial{y}^2}+\frac{2}{y}\frac{\partial}{\partial{y}}\big)\big(\mathcal{K}(\alpha;\frac{1}{2}+iy)-2 \mathcal{K}(2\alpha;\frac{1}{2}+iy)\big)\geq \frac{4\pi\alpha}{y^4}e^{-\frac{\pi\alpha}{y}}\big(A(\alpha;y)+B(\alpha;y)\big),
\endaligned\end{equation}
where
\begin{equation}\aligned\nonumber
A(\alpha;y):=&\frac{\pi\alpha}{2y}-(1+\epsilon_{b})-4e^{-\frac{\pi\alpha}{y}}\big(\frac{\pi\alpha}{y}(1+\epsilon_{d})-1\big),\:\:B(\alpha;y):=B_{1}(\alpha;y)-B_{2}(\alpha;y),
\endaligned\end{equation}
and
\begin{equation}\aligned\nonumber
B_{1}(\alpha;y):=&2y^4e^{-\pi\alpha(y-\frac{3}{4y})}\big(\frac{\pi\alpha y}{2}(1-\frac{1}{4y^2})^2(1+\frac{1}{4y^2})+\frac{1}{\pi\alpha y}-1-\epsilon_{c}\big),\\
B_{2}(\alpha;y):=&8y^4e^{-2\pi\alpha(y-\frac{1}{4y})}\big(\pi\alpha y(1+\epsilon_{e} -\frac{1}{4y^2} -\frac{1}{16y^4})+\frac{1}{2\pi\alpha y}(1+\epsilon_{a})-(1-\frac{1}{4y^2})^2-1-\frac{1}{4y^2}\big).
\endaligned\end{equation}
Here,
\begin{equation}\aligned\nonumber
\epsilon_{a}\leq4\cdot 10^{-6},\:\epsilon_{b}\leq6\cdot 10^{-3},\:\epsilon_{c}\leq4\cdot 10^{-1},\:\epsilon_{d}\leq 5\cdot10^{-7},\:\epsilon_{e}\leq 4\cdot 10^{-2}.
\endaligned\end{equation}
Note that $ \alpha \geq 2$ and $y\in[\frac{\sqrt{3}}{2},2]$ implies that $\frac{\alpha}{y}\geq 1$.
Using this observation along with an elementary inequality,
\begin{equation}\aligned\label{4.1eq2}
\frac{\pi}{2}x-(1+\epsilon_{b})-4e^{-\pi x}\big(\pi x(1+\epsilon_{d})-1\big)\geq10^{-1},\:\:\:\:\mathrm{for}\:\:x\geq1,
\endaligned\end{equation}
we obtain
\begin{equation}\aligned\label{4.1eq3}
A(\alpha;y)\geq 10^{-1}.
\endaligned\end{equation}
For $B(\alpha;y)$, given that $\alpha \geq 2, y\in[\frac{\sqrt{3}}{2},2]$, one has
\begin{equation}\aligned\label{4.1eq4}
B_{1}(\alpha;y)\geq0,\:\:B_{2}(\alpha;y)\leq B_{2}(2,\frac{\sqrt{3}}{2})\leq 5\cdot 10^{-3}.
\endaligned\end{equation}
Therefore, by \eqref{4.1eq1} and \eqref{4.1eq4}, we obtain
\begin{equation}\aligned\nonumber
A(\alpha;y)+B(\alpha;y)\geq10^{-1}-5\cdot 10^{-3}>0,
\endaligned\end{equation}
which yields the result.
\end{proof}

\subsection{Proof of Proposition \ref{4prop2}}

We first give the exponential expansion of $\mathcal{K}(\alpha;\frac{1}{2}+iy)-2 \mathcal{K}(2\alpha;\frac{1}{2}+iy)$.
 \begin{lemma}\label{4lemma0}For $\alpha,y>0$, we have the following expression of $(\mathcal{K}(\alpha;\frac{1}{2}+iy)-2 \mathcal{K}(2\alpha;\frac{1}{2}+iy)):$
\begin{equation}\aligned\nonumber
\mathcal{K}(\alpha;\frac{1}{2}+iy)-2 \mathcal{K}(2\alpha;\frac{1}{2}+iy)=&\frac{1}{\pi}{2}^{-\frac{5}{2}}\alpha^{-\frac{5}{2}}y^{\frac{1}{2}}
\Big(\sum_{n\in\mathbb{Z}}2\sqrt{2} \alpha(1+2\pi\alpha y {n}^2) e^{-\pi\alpha y n^2}\vartheta(\frac{y}{\alpha};\frac{n}{2})\\
&+\sum_{n\in\mathbb{Z}} 4\sqrt{2}ye^{-\pi\alpha y n^2}\vartheta_X(\frac{y}{\alpha};\frac{n}{2})- \sum_{n\in\mathbb{Z}}2 ye^{-2\pi\alpha y n^2}\vartheta_X(\frac{y}{2 \alpha};\frac{n}{2})\\
&- \sum_{n\in\mathbb{Z}}2 \alpha(1+4\pi\alpha y{n}^2) e^{-2\pi\alpha  y n^2}\vartheta(\frac{y}{2 \alpha};\frac{n}{2})\Big).\\
\endaligned\end{equation}
\end{lemma}

In view of Lemma \ref{4lemma0}, to better demonstrate the proof of Proposition \ref{4prop2}, we will decompose the expression $\mathcal{K}(\alpha;\frac{1}{2}+iy)-2 \mathcal{K}(2\alpha;\frac{1}{2}+iy)$. We denote that

 \begin{equation}\aligned\label{4eq4}
M_{1}(\alpha;y):=&2\sqrt{2}\alpha\:\vartheta(\frac{y}{\alpha};0)-2\alpha\:\vartheta(\frac{y}{2\alpha};0),\:\:\:\:
M_{2}(\alpha;y):=4\sqrt{2}y\:\vartheta_X(\frac{y}{\alpha};0)-2 y\:\vartheta_X(\frac{y}{2 \alpha};0),\\
M_{3}(\alpha;y):=&4\sqrt{2}\alpha(1+2\pi\alpha y)e^{-\pi\alpha y}\vartheta(\frac{y}{\alpha};\frac{1}{2})-4\alpha(1+4\pi\alpha y)e^{-2\pi\alpha y}\vartheta(\frac{y}{2\alpha};\frac{1}{2}),\\
M_{4}(\alpha;y):=&8\sqrt{2}y e^{-\pi\alpha y} \vartheta_X(\frac{y}{\alpha};\frac{1}{2})-4y e^{-2\pi\alpha y}\vartheta_X(\frac{y}{2 \alpha};\frac{1}{2}),\\
\endaligned\end{equation}
and
\begin{equation}\aligned\label{4eq5}
\mathcal{E}_{1}(\alpha;y):=&4\sqrt{2} \alpha\sum_{n\geq2}(1+2\pi\alpha y {n}^2) e^{-\pi\alpha y n^2}\vartheta(\frac{y}{\alpha};\frac{n}{2}),\:\:\:\:\:\:
\mathcal{E}_{2}(\alpha;y):=8\sqrt{2}y\sum_{n\geq2} e^{-\pi\alpha y n^2}\vartheta_X(\frac{y}{\alpha};\frac{n}{2}),\\
\mathcal{E}_{3}(\alpha;y):=&-4 \alpha \sum_{n\geq2}(1+4\pi\alpha y {n}^2) e^{-2\pi\alpha  y n^2}\vartheta(\frac{y}{2 \alpha};\frac{n}{2}),\:\:\:
\mathcal{E}_{4}(\alpha;y):=-4 y\sum_{n\geq2} e^{-2\pi\alpha y n^2}\vartheta_X(\frac{y}{2 \alpha};\frac{n}{2}).
\endaligned\end{equation}
By \eqref{4eq4} and \eqref{4eq5}, we further denote that
 \begin{equation}\aligned\label{4eq3}
M(\alpha;y):=\underset{i=1}{\overset{4}{\sum}}M_{i}(\alpha;y),\:\:\mathcal{E}(\alpha;y):=\underset{i=1}{\overset{4}{\sum}}\mathcal{E}_{i}(\alpha;y).
\endaligned\end{equation}
Thus, by Lemma \ref{4lemma0} and \eqref{4eq4}-\eqref{4eq3}, we have
\begin{equation}\aligned\label{4eq2}
\mathcal{K}(\alpha;\frac{1}{2}+iy)-2 \mathcal{K}(2\alpha;\frac{1}{2}+iy)=\frac{1}{\pi}{2}^{-\frac{5}{2}}\alpha^{-\frac{5}{2}}y^{\frac{1}{2}}\big(M(\alpha;y)+\mathcal{E}(\alpha;y)\big).\\
\endaligned\end{equation}

By \eqref{4eq2}, Proposition \ref{4prop2} is equivalent to the following lemma.

\begin{lemma}\label{4lemma}Assume that $ \alpha\geq 2,\;y\geq2$, it holds that
\begin{equation}\aligned\nonumber
y^{\frac{1}{2}}\big(M(\alpha;y)+ \mathcal{E}(\alpha;y)\big)-(\frac{\sqrt{3}}{2})^{\frac{1}{2}}\big(M(\alpha;\frac{\sqrt{3}}{2})+ \mathcal{E}(\alpha;\frac{\sqrt{3}}{2})\big)>0.
\endaligned\end{equation}
\end{lemma}

To prove Lemma \ref{4lemma}, we aim to show that the term
\begin{equation}\nonumber
y^{\frac{1}{2}}M(\alpha;y) - (\frac{\sqrt{3}}{2})^{\frac{1}{2}}M(\alpha;\frac{\sqrt{3}}{2})
\end{equation}
is positive and serves as the principal term, while the term
\begin{equation}\nonumber
y^{\frac{1}{2}}\mathcal{E}(\alpha;y) - (\frac{\sqrt{3}}{2})^{\frac{1}{2}}\mathcal{E}(\alpha;\frac{\sqrt{3}}{2})
\end{equation}
is significantly smaller in comparison when when $\alpha,y\geq2$. Given the complexity of this problem, our proof will be divided into two cases: $\frac{y}{\alpha}\geq1$ and $\frac{y}{\alpha}\in(0,1)$.

The proof of cases $\frac{y}{\alpha}\geq 1$ and $\frac{y}{\alpha}\in(0,1)$ will be given in Lemma \ref{4.2lemma2} and \ref{4.3lemma2}, respectively.
\begin{lemma}\label{4.2lemma2}Assume that $ y\geq\alpha \geq 2$. Then
\begin{itemize}
  \item [(1)]
  $ y^{\frac{1}{2}} M(\alpha;y)-(\frac{\sqrt{3}}{2})^{\frac{1}{2}}M(\alpha;\frac{\sqrt{3}}{2})\geq \frac{1}{500} \sqrt{y}\alpha>0.$
\item [(2)]
  $\left|\frac{ \mathcal{E}(\alpha;y)}{\alpha}\right|+\frac{{3}^{\frac{1}{4}}}{2}\left|\frac{\mathcal{E}(\alpha;\frac{\sqrt{3}}{2}) }{\alpha}\right|\leq10^{-6}.$
    \item [(3)]    $\left|\frac{ y^{\frac{1}{2}}\mathcal{E}(\alpha;y)-(\frac{\sqrt{3}}{2})^{\frac{1}{2}}\mathcal{E}(\alpha;\frac{\sqrt{3}}{2})}{y^{\frac{1}{2}}M(\alpha;y)-(\frac{\sqrt{3}}{2})^{\frac{1}{2}}M(\alpha;\frac{\sqrt{3}}{2})}\right|
\leq 10^{-3}.$
\item [(4)]
$y^{\frac{1}{2}}\big(M(\alpha;y)+ \mathcal{E}(\alpha;y)\big)-(\frac{\sqrt{3}}{2})^{\frac{1}{2}}\big(M(\alpha;\frac{\sqrt{3}}{2})+ \mathcal{E}(\alpha;\frac{\sqrt{3}}{2})\big)\geq \frac{1}{500} \sqrt{y}\alpha\:(1-10^{-3})>0.$
         \end{itemize}

\end{lemma}
\begin{proof}
(1). Notice that $y\geq\alpha \geq 2$. By Lemmas \ref{lemma4.19} and \ref{lemma4.20}, we have
 \begin{equation}\aligned\nonumber
\sqrt{y} M(\alpha;y)-(\frac{\sqrt{3}}{2})^{\frac{1}{2}}M(\alpha;\frac{\sqrt{3}}{2})\geq \sqrt{y}\alpha\big(\frac{1}{5} -y^{-\frac{1}{2}}(\frac{\sqrt{3}}{2})^{\frac{1}{2}}\cdot \frac{3}{10}\big)\geq\frac{1}{500} \sqrt{y}\alpha.
\endaligned\end{equation}

(2). By \eqref{4eq5} and Lemma \ref{lem4.23}, we have
\begin{equation}\aligned\nonumber
\left|\frac{\mathcal{E}_{1}(\alpha;y) }{\alpha}\right|
=&4\sqrt{2}e^{-4\pi\alpha y}\vartheta(\frac{y}{\alpha};1)\sum_{n\geq2}(1+2\pi\alpha y {n}^2) e^{-\pi\alpha y (n^2-4)}\left|\frac{\vartheta(\frac{y}{\alpha};\frac{n}{2}}{\vartheta(\frac{y}{\alpha};1)})\right|\\
\leq& 4\sqrt{2}e^{-4\pi\alpha y}\vartheta(\frac{y}{\alpha};1)\sum_{n\geq2}(1+2\pi\alpha y {n}^2) e^{-\pi\alpha y (n^2-4)}\\
\leq& 4\sqrt{2}e^{-4\pi\alpha y}\vartheta(\frac{y}{\alpha};1)\Big(1+8\pi\alpha y +\sum_{n\geq3}(1+2\pi\alpha y {n}^2) e^{-\pi\alpha y (n^2-4)}\Big).\\
\endaligned\end{equation}
Given that $y\geq \alpha\geq2$, we recall from \eqref{TXY} that $$\vartheta(\frac{y}{\alpha};1)=1+2\underset{n\geq1}{\sum}e^{-\pi n^2\frac{y}{\alpha}}\leq 1+2\underset{n\geq1}{\sum}e^{-\pi n^2}.$$
Trivially, we have the following bound
$$\sum_{n\geq1} e^{-\pi n^2}\leq4.33\cdot 10^{-2},\;\sum_{n\geq3}(1+8\pi {n}^2) e^{-4 \pi (n^2-4)}\leq 10^{-40}.$$
Therefore, combining these together, we obtain that
\begin{equation}\aligned\label{M2aaa}
\left|\frac{\mathcal{E}_{1}(\alpha;y) }{\alpha}\right|\leq \frac{31}{5}(1+8\pi\alpha y)e^{-4\pi\alpha y}.
\endaligned\end{equation}
Similarly, to avoid repetition, we have
\begin{equation}\aligned\label{M2aab}
\left|\frac{\mathcal{E}(\alpha;y)}{\alpha}\right|\leq &\underset{i=1}{\overset{4}{\sum}}\left|\frac{\mathcal{E}_{i}(\alpha;y)}{\alpha}\right|\leq2\pi(1+8\pi\alpha y+\frac{y}{2\alpha})e^{-4\pi\alpha y},\\
\left|\frac{\mathcal{E}(\alpha;\frac{\sqrt{3}}{2})}{\alpha}\right|\leq & \underset{j=1}{\overset{4}{\sum}}\left|\frac{\mathcal{E}_{j}(\alpha;\frac{\sqrt{3}}{2})}{\alpha}\right|\leq4\pi{\alpha}^{\frac{1}{2}}(1+2\sqrt{3}\pi\alpha)e^{-2\sqrt{3}\pi\alpha}.
\endaligned\end{equation}
\eqref{M2aaa} and \eqref{M2aab} yield (2).

(3). Noting $y\geq2$, by Lemma \ref{4.2lemma2}, one has
\begin{equation}\aligned\label{4.2lemaddeq1}
\left|\frac{ y^{\frac{1}{2}}\mathcal{E}(\alpha;y)-(\frac{\sqrt{3}}{2})^{\frac{1}{2}}\mathcal{E}(\alpha;\frac{\sqrt{3}}{2})}{y^{\frac{1}{2}}M(\alpha;y)-(\frac{\sqrt{3}}{2})^{\frac{1}{2}}M(\alpha;\frac{\sqrt{3}}{2})}\right|
\leq&\left| \frac{y^{\frac{1}{2}}\mathcal{E}(\alpha;y)-(\frac{\sqrt{3}}{2})^{\frac{1}{2}}\mathcal{E}(\alpha;\frac{\sqrt{3}}{2}) }{\frac{1}{500}\sqrt{y}\alpha}\right|\\
\leq& 500\bigg(\left|\frac{ \mathcal{E}(\alpha;y)}{\alpha}\right|+(\frac{\sqrt{3}}{2})^{\frac{1}{2}}y^{-\frac{1}{2}}\left|\frac{\mathcal{E}(\alpha;\frac{\sqrt{3}}{2}) }{\alpha}\right|\bigg)\\
\leq& 500\bigg(\left|\frac{ \mathcal{E}(\alpha;y)}{\alpha}\right|+\frac{{3}^{\frac{1}{4}}}{2}\left|\frac{\mathcal{E}(\alpha;\frac{\sqrt{3}}{2}) }{\alpha}\right|\bigg).
\endaligned\end{equation}
Then \eqref{4.2lemaddeq1} and item (2) yield item (3). Items (1) and (3) yield item (4).
\end{proof}

Recall that $M(\alpha;y)=\underset{i=1}{\overset{4}{\sum}}M_{i}(\alpha;y).$ In the following lemma, we will provide an lower bound for $M(\alpha;y)$ to prove (1) of Lemma \ref{4.2lemma2}.

\begin{lemma}\label{lemma4.19} Assume that $ y\geq \alpha \geq2,$ then $M(\alpha;y)\geq\frac{1}{5}\alpha$. We split it into four items.
\begin{itemize}
  \item [(1)] $M_{1}(\alpha;y)\geq \frac{1}{5}\alpha.$
      \item [(2)] $M_{2}(\alpha;y)>0.$
      \item [(3)] $M_{3}(\alpha;y)>0.$
      \item [(4)] $M_{4}(\alpha;y)>0.$
   \end{itemize}
\end{lemma}
\begin{proof}
(1). By \eqref{TXY} and \eqref{4eq4}, one has
\begin{equation}\aligned\nonumber
 M_{1}(\alpha;y)
 = (2\sqrt{2}-2)\alpha-4\alpha\underset{n=1}{\overset{\infty}{\sum}}e^{-\pi n^2\frac{y}{2\alpha}}(1-\sqrt{2}e^{-\pi n^2\frac{y}{2\alpha}}).
\endaligned\end{equation}
Since $\frac{y}{\alpha}\geq 1$, one gets
\begin{equation}\aligned\nonumber
\underset{n=1}{\overset{\infty}{\sum}}e^{-\pi n^2\frac{y}{2\alpha}}(1-\sqrt{2}e^{-\pi n^2\frac{y}{2\alpha}})\leq \underset{n=1}{\overset{\infty}{\sum}}e^{-\frac{\pi n^2}{2}}(1-\sqrt{2}e^{-\frac{\pi n^2}{2}})=0.1486\cdots.
\endaligned\end{equation}
(2). For $M_{2}(\alpha;y)$, by \eqref{TXY}, one has
\begin{equation}\aligned\label{lem4.17eq1}
{\vartheta}_{X}(X;Y)=-2\pi\underset{n=1}{\overset{\infty}{\sum}}n^2e^{-\pi n^2 X}\cos(2n\pi Y).
\endaligned\end{equation}
By \eqref{4eq4}, \eqref{lem4.17eq1} and using $\frac{y}{\alpha}\geq1$, one has
\begin{equation}\aligned\nonumber
 M_{2}(\alpha;y)=&4\pi y\underset{n=1}{\overset{\infty}{\sum}}n^2 e^{-\pi n^2 \frac{y}{2\alpha}}(1-2\sqrt{2} e^{-\pi n^2 \frac{y}{2\alpha}})\geq
 4\pi y\underset{n=1}{\overset{\infty}{\sum}}n^2 e^{-\pi n^2 \frac{y}{2\alpha}}(1-2\sqrt{2} e^{-\frac{\pi }{2}})>0.
\endaligned\end{equation}
(3). For $M_{3}(\alpha;y)$, by \eqref{TXY} and \eqref{4eq4}, one has
\begin{equation}\aligned\nonumber
 M_{3}(\alpha;y)=&4\sqrt{2}\alpha(1+2\pi\alpha y)e^{-\pi\alpha y}(1+2\underset{n=1}{\overset{\infty}{\sum}}e^{-\pi n^2\frac{y}{\alpha}}(-1)^{n})\\
& -4\alpha(1+4\pi\alpha y)e^{-2\pi\alpha y}(1+2\underset{n=1}{\overset{\infty}{\sum}}e^{-\pi n^2\frac{y}{2\alpha}}(-1)^{n})\\
\geq& 4\sqrt{2}\alpha(1+2\pi\alpha y)e^{-\pi\alpha y}(1-2e^{-\pi\frac{y}{\alpha}})-4\alpha(1+4\pi\alpha y)e^{-2\pi\alpha y}.
\endaligned\end{equation}
Notice that $ y\geq\alpha \geq 2$. Then $M_{3}(\alpha;y)>0$.

(4). For $ M_{4}(\alpha;y)$, by \eqref{4eq4} and \eqref{lem4.17eq1}, one has
\begin{equation}\aligned\nonumber
 M_{4}(\alpha;y) =&16\sqrt{2}\pi y e^{-\pi\alpha y} \underset{n=1}{\overset{\infty}{\sum}}n^2 e^{-\pi n^2 \frac{y}{\alpha}}(-1)^{n+1}+8\pi y e^{-2\pi\alpha y}\underset{n=1}{\overset{\infty}{\sum}}n^2 e^{-\pi n^2 \frac{y}{2\alpha}}(-1)^{n}\\
 \geq &16\sqrt{2}\pi y e^{-\pi y(\alpha+\frac{1}{\alpha})}(1-4 e^{-3\pi \frac{y}{\alpha}}-\frac{\sqrt{2}}{4}e^{-\pi y(\alpha-\frac{1}{2\alpha})}).
\endaligned\end{equation}
Note that $y\geq\alpha \geq 2$, then $e^{-3\pi \frac{y}{\alpha}}\leq e^{-3\pi },\:\:e^{-\pi y(\alpha-\frac{1}{2\alpha})}\leq e^{-\frac{7}{2}\pi}$. Thus $ M_{4}(\alpha;y)>0.$
\end{proof}

Before providing an upper bound function of $M(\alpha;\frac{\sqrt{3}}{2})$, we denote that
\begin{equation}\aligned\label{lem4.15eq0}
 \epsilon_{1}:=&\underset{n=3}{\overset{\infty}{\sum}}e^{-\frac{2\sqrt{3}\pi\alpha (n^2-4)}{3}},\:\:\:\:\:\:\:\:\:\:\:\:\:\:\:\:\:\:\:\:\:
 \epsilon_{2}:=\underset{n=1}{\overset{\infty}{\sum}}e^{-\frac{2\sqrt{3}\pi\alpha(n-\frac{1}{2})^2}{3}},\\
 \epsilon_{3}:=&\underset{n=2}{\overset{\infty}{\sum}}n^2e^{-\frac{2\sqrt{3}\pi\alpha ({n}^2-1)}{3}},\:\:\:\:\:\:\:\:\:\:\:\:\:\:\:\:\:
  \epsilon_{4}:=\underset{n=2}{\overset{\infty}{\sum}}4(n-\frac{1}{2})^2e^{-\frac{2\sqrt{3}\pi\alpha(n-1)n}{3}}.\\
\endaligned\end{equation}

\begin{lemma}[An upper bound function of $M(\alpha;\frac{\sqrt{3}}{2})$]\label{lemma4.20}Assume that $ \alpha \geq 2$. Then
\begin{itemize}
  \item [(1)] $M_{1}(\alpha;\frac{\sqrt{3}}{2})\leq 8{\alpha}^{\frac{3}{2}}3^{-\frac{1}{4}}\big(e^{-\frac{2\sqrt{3}\pi\alpha }{3}}-e^{-\frac{4\sqrt{3}\pi\alpha}{3}}+e^{-\frac{8\sqrt{3}\pi\alpha }{3}}(1+\epsilon_{1})\big).$
      \item [(2)] $M_{2}(\alpha;\frac{\sqrt{3}}{2})
  \leq 8{\alpha}^{\frac{3}{2}}3^{-\frac{1}{4}}(e^{-\frac{4\sqrt{3}\pi\alpha} {3}}-e^{-\frac{2\sqrt{3}\pi\alpha}{3}})
+32\pi{\alpha}^{\frac{5}{2}}3^{-\frac{3}{4}}\big(e^{-\frac{2\sqrt{3}\pi\alpha }{3}}(1+\epsilon_{3})-2e^{-\frac{4\sqrt{3}\pi\alpha}{3}}\big).$
      \item [(3)] $M_{3}(\alpha;\frac{\sqrt{3}}{2})
  \leq 16{\alpha}^{\frac{3}{2}}3^{-\frac{1}{4}}e^{-\frac{\sqrt{3}}{2}\pi\alpha}\big( (1+\sqrt{3}\pi\alpha)\epsilon_{2}-(1+2\sqrt{3}\pi\alpha)e^{-\frac{5\sqrt{3}\pi\alpha}{6}}\big).$
      \item [(4)] $M_{4}(\alpha;\frac{\sqrt{3}}{2})\leq 16\pi{\alpha}^{\frac{5}{2}}3^{-\frac{3}{4}}e^{-\frac{2\sqrt{3}\pi\alpha}{3}}(1+\epsilon_{4}).$
      \item [(5)] $M(\alpha;\frac{\sqrt{3}}{2})\leq \frac{792}{25}{\alpha}^{\frac{3}{2}}3^{-\frac{1}{4}}e^{-\frac{\sqrt{3}}{2}\pi\alpha}\big( 1+\frac{\sqrt{3}\pi\alpha}{11}\big).$ Trivially, $M(\alpha;\frac{\sqrt{3}}{2})\leq\frac{3}{10}\alpha.$
   \end{itemize}
\end{lemma}
\begin{proof}
(1). For $M_{1}(\alpha;\frac{\sqrt{3}}{2})$, by \eqref{Poisson} and \eqref{4eq4}, one has
\begin{equation}\aligned\nonumber
 M_{1}(\alpha;\frac{\sqrt{3}}{2})=&8 {\alpha}^{\frac{3}{2}}3^{-\frac{1}{4}}\underset{n=1}{\overset{\infty}{\sum}}(e^{-\frac{2\sqrt{3}\pi\alpha n^2}{3}}-e^{-\frac{4\sqrt{3}\pi\alpha n^2}{3}})\\
  \leq &8 {\alpha}^{\frac{3}{2}}3^{-\frac{1}{4}}(e^{-\frac{2\sqrt{3}\pi\alpha }{3}}+e^{-\frac{8\sqrt{3}\pi\alpha }{3}}(1+\epsilon_{1})-e^{-\frac{4\sqrt{3}\pi\alpha}{3}}).\\
\endaligned\end{equation}
(2). For $M_{2}(\alpha;\frac{\sqrt{3}}{2})$, by \eqref{Poisson}, one has
\begin{equation}\aligned\label{lem4.22eq1}
\vartheta_{X}(X;Y)=-\frac{1}{2}X^{-\frac{3}{2}}\sum_{n\in\mathbb{Z}} e^{-\pi \frac{(n-Y)^2}{X}} +\pi {X}^{-\frac{5}{2}}\sum_{n\in\mathbb{Z}} (n-Y)^2e^{-\pi \frac{(n-Y)^2}{X}}.
\endaligned\end{equation}
Then by \eqref{4eq4} and \eqref{lem4.22eq1}, we have
\begin{equation}\aligned\nonumber
M_{2}(\alpha;\frac{\sqrt{3}}{2})=&8{\alpha}^{\frac{3}{2}}3^{-\frac{1}{4}}\underset{n=1}{\overset{\infty}{\sum}}(e^{-\frac{4\sqrt{3}\pi\alpha {n}^2}{3}}-e^{-\frac{2\sqrt{3}\pi\alpha {n}^2}{3}})\\
&+32\pi{\alpha}^{\frac{5}{2}}3^{-\frac{3}{4}}\underset{n=1}{\overset{\infty}{\sum}}(n^2e^{-\frac{2\sqrt{3}\pi\alpha {n}^2}{3}}-2 n^2e^{-\frac{4\sqrt{3}\pi\alpha {n}^2}{3}})\\
\leq& 8{\alpha}^{\frac{3}{2}}3^{-\frac{1}{4}}(e^{-\frac{4\sqrt{3}\pi\alpha} {3}}-e^{-\frac{2\sqrt{3}\pi\alpha}{3}})+32\pi{\alpha}^{\frac{5}{2}}3^{-\frac{3}{4}}\big(e^{-\frac{2\sqrt{3}\pi\alpha }{3}}(1+\epsilon_{3})-2e^{-\frac{4\sqrt{3}\pi\alpha}{3}}\big).
\endaligned\end{equation}
(3). For $M_{3}(\alpha;\frac{\sqrt{3}}{2})$, by \eqref{Poisson} and \eqref{4eq4}, we have
  \begin{equation}\aligned\label{lem4.22eq2}
 M_{3}(\alpha;\frac{\sqrt{3}}{2})=&16{\alpha}^{\frac{3}{2}}3^{-\frac{1}{4}}e^{-\frac{\sqrt{3}}{2}\pi\alpha}\big((1+\sqrt{3}\pi\alpha)\cdot\epsilon_{2}
 -(1+2\sqrt{3}\pi\alpha)e^{-\frac{\sqrt{3}\pi\alpha}{2}}\underset{n=1}{\overset{\infty}{\sum}}e^{-\frac{4\sqrt{3}\pi\alpha(n-\frac{1}{2})^2}{3}}\big),
  \endaligned\end{equation}
  which yields the result.\\
(4). For $M_{4}(\alpha;\frac{\sqrt{3}}{2})$,
by \eqref{4eq4} and \eqref{lem4.22eq1}, we have
 \begin{equation}\aligned\nonumber
 M_{4}(\alpha;\frac{\sqrt{3}}{2})=&64\pi{\alpha}^{\frac{5}{2}}{3}^{-\frac{3}{4}}e^{-\frac{\sqrt{3}\pi\alpha}{2}}(\underset{n=1}{\overset{\infty}{\sum}}(n-\frac{1}{2})^{2}e^{-\frac{2\sqrt{3}\pi\alpha(n-\frac{1}{2})^2}{3}}-\frac{\sqrt{3}}{4\pi\alpha}\underset{n=1}{\overset{\infty}{\sum}}e^{-\frac{2\sqrt{3}\pi\alpha(n-\frac{1}{2})^2}{3}})\\
 &+16 {\alpha}^{\frac{3}{2}}{3}^{-\frac{1}{4}}e^{-\sqrt{3}\pi\alpha}\underset{n=1}{\overset{\infty}{\sum}}(1-\frac{8\sqrt{3}\pi\alpha(n-\frac{1}{2})^2}{3})e^{-\frac{4\sqrt{3}\pi\alpha(n-\frac{1}{2})^2}{3}}\\
 \leq&64\pi{\alpha}^{\frac{5}{2}}{3}^{-\frac{3}{4}}e^{-\frac{\sqrt{3}\pi\alpha}{2}}\underset{n=1}{\overset{\infty}{\sum}}(n-\frac{1}{2})^{2}e^{-\frac{2\sqrt{3}\pi\alpha(n-\frac{1}{2})^2}{3}}
 \leq 16\pi{\alpha}^{\frac{5}{2}}3^{-\frac{3}{4}}e^{-\frac{2\sqrt{3}\pi\alpha}{3}}(1+\epsilon_{4}).
 \endaligned\end{equation}
 (5). Combining items (1)-(4) and noting that $\alpha\geq2$, one has item (5).
           \end{proof}

We then give the proof of item (2) in Lemma \ref{4.2lemma2}.

\begin{lemma}\label{lem4.23} Assume that $n\in\mathbb{Z}$. Then it holds that
\begin{itemize}
  \item [(1)]If $X>0$, then
  $\left|\frac{\vartheta(X;\frac{n}{2})}{\vartheta(X;1)}\right|\leq1.$
  \item [(2)] If $X>0$, then
  $\left|\frac{\vartheta_{X}(X;\frac{n}{2})}{\vartheta_{X}(X;1)}\right|\leq 2.$
     \end{itemize}
\end{lemma}
\begin{proof}
For simplicity, we denote that
\begin{equation}\aligned\label{lem4.23eq1}
f(X;Y):=\underset{m\in \mathbb{Z}}{\sum}e^{-\frac{\pi(m-Y)^2}{X}},\:\:g(X;Y):=\underset{m\in \mathbb{Z}}{\sum}(m-Y)^2e^{-\frac{\pi(m-Y)^2}{X}}.
\endaligned\end{equation}
A direct calculation shows that
\begin{equation}\aligned\label{lem4.23eq1add}
f(X;Y)=f(X;Y+1),\:\:g(X;Y)=g(X;Y+1).
\endaligned\end{equation}
As $X\in (0,1]$, by \eqref{lem4.23eq1}, one has
\begin{equation}\aligned\label{lem4.23eq2}
\frac{f(X;\frac{1}{2})}{f(X;1)}=\frac{2\underset{m=1}{\overset{\infty}{\sum}}e^{-\frac{\pi(m-\frac{1}{2})^2}{X}}}
{1+2\underset{m=2}{\overset{\infty}{\sum}}e^{-\frac{\pi(m-1)^2}{X}}}\leq 2\underset{m=1}{\overset{\infty}{\sum}}e^{-\pi(m-\frac{1}{2})^2}<1.
\endaligned\end{equation}
By \eqref{Poisson} and \eqref{lem4.23eq1add}-\eqref{lem4.23eq2}, one has
\begin{equation}\aligned\label{lem4.23eq3}
\left|\frac{\vartheta(X;\frac{n}{2})}{\vartheta(X;1)}\right|
=\frac{f(X;\frac{n}{2})}{f(X;1)}
=\left\{
   \begin{array}{ll}
   \frac{f(X;\frac{1}{2})}{f(X;1)},&   \hbox{if}\:\:\:n\:\: is \:\:odd;\\
     \:\:\:\:\:1,& \hbox{if}\:\:\:n\:\:is \:\:even.\\
        \end{array}
 \right.
\endaligned\end{equation}
Given that $X\geq1$, by \eqref{TXY}, one has
\begin{equation}\aligned\label{lem4.23eq3add}
\left|\frac{\vartheta(X;\frac{n}{2})}{\vartheta(X;1)}\right|=\frac{1+2\sum_{m=1}^{\infty}e^{-\pi m^2 X} \cos (mn\pi)}{1+2\sum_{m=1}^{\infty}e^{-\pi m^2 X}}\leq 1.
\endaligned\end{equation}
Then \eqref{lem4.23eq2}-\eqref{lem4.23eq3add} yield item (1). By \eqref{lem4.22eq1}, one has
\begin{equation}\aligned\label{lem4.23eq4}
\vartheta_{X}(X;Y)=-\frac{1}{2}X^{-\frac{3}{2}}f(X;Y)+\pi{X}^{-\frac{5}{2}}g(X;Y).
 \endaligned\end{equation}
 Then for $X\in(0,1],n\in\mathbb{Z}$, by \eqref{lem4.23eq1add} and \eqref{lem4.23eq4}, we have
 \begin{equation}\aligned\label{lem4.23eq5}
\left|\frac{\vartheta_{X}(X;\frac{n}{2})}{\vartheta_{X}(X;1)}\right|
=\left|\frac{\vartheta_{X}(X;\frac{n}{2}+1)}{\vartheta_{X}(X;1)}\right|
=\left\{
   \begin{array}{ll}
    \left|\frac{\vartheta_{X}(X;\frac{1}{2})}{\vartheta_{X}(X;1)}\right|,&  \hbox{if}\:\:\:n\:\: is \:\:odd ;\\
     \:\:\:\:\:1,& \hbox{if}\:\:\:n\:\:is \:\:even.\\
        \end{array}
 \right.
  \endaligned\end{equation}
  By \eqref{lem4.22eq1}, one has
\begin{equation}\aligned\label{lem4.23add0eq1}
\left|\vartheta_{X}(X;\frac{1}{2})\right|=& 2\pi {X}^{-\frac{5}{2}}\sum_{n=1}^{\infty} (n-\frac{1}{2})^2e^{-\pi \frac{(n-\frac{1}{2})^2}{X}}-X^{-\frac{3}{2}}\sum_{n=1}^{\infty} e^{-\pi \frac{(n-\frac{1}{2})^2}{X}}\\
\leq&\frac{\pi}{2}X^{-\frac{5}{2}}e^{-\frac{\pi}{4X}}(1+4\sum_{n=2}^{\infty} (n-\frac{1}{2})^2e^{-\pi \frac{(n-1)\cdot n}{X}})-X^{-\frac{3}{2}}e^{-\frac{\pi}{4X}},
\endaligned\end{equation}
and
\begin{equation}\aligned\label{lem4.23add0eq2}
\left|\vartheta_{X}(X;1)\right|=\frac{1}{2}X^{-\frac{3}{2}}\Big(1-\sum_{n=2}^{\infty}\big(4\pi(n-1)^2X^{-1}-2\big)e^{-\frac{\pi(n-1)^2}{X}}\Big).
\endaligned\end{equation}
Here, since $X\in(0,1]$, one has
\begin{equation}\aligned\label{lem4.23add0eq3}
4\sum_{n=2}^{\infty} (n-\frac{1}{2})^2e^{-\pi \frac{(n-1)\cdot n}{X}}\leq4\sum_{n=2}^{\infty} (n-\frac{1}{2})^2e^{-\pi\cdot(n-1)\cdot n}\leq \frac{1}{50},
\endaligned\end{equation}
and
\begin{equation}\aligned\label{lem4.23add0eq4}
\sum_{n=2}^{\infty}\big(4\pi(n-1)^2X^{-1}-2\big)e^{-\frac{\pi(n-1)^2}{X}}\leq \sum_{n=2}^{\infty}\big(4\pi(n-1)^2-2\big)e^{-\pi(n-1)^2}\leq \frac{1}{2}.
\endaligned\end{equation}
Then for $X\in(0,1]$, by \eqref{lem4.23add0eq1}-\eqref{lem4.23add0eq4}, one has
\begin{equation}\aligned\label{lem4.23add0eq4add}
\left|\frac{\vartheta_{X}(X;\frac{1}{2})}{\vartheta_{X}(X;1)}\right|\leq e^{-\frac{\pi}{4X}}\big(2\pi{X}^{-1}(1+\frac{1}{50}) -4 \big)\leq2.
\endaligned\end{equation}
 As $X\geq1,\;n\in\mathbb{Z}$, by \eqref{TXY}, one has
  \begin{equation}\aligned\label{lem4.23eq5add}
\left|\frac{\vartheta_{X}(X;\frac{n}{2})}{\vartheta_{X}(X;1)}\right|=\left|\frac{\sum_{m=1}^{\infty}m^2e^{-\pi m^2 X} \cos (mn\pi)}{\sum_{m=1}^{\infty}m^2e^{-\pi m^2 X}}\right|\leq1.
 \endaligned\end{equation}
 Then combining \eqref{lem4.23eq5} with \eqref{lem4.23add0eq4add}-\eqref{lem4.23eq5add}, we obtain item (2).

\end{proof}

Finally, we provide the proof Lemma \ref{4lemma} for $\frac{\alpha}{y}\geq1$. It is stated as follows:
\begin{lemma}\label{4.3lemma2}Assume that $ \alpha\geq y\geq 2$. Then
\begin{itemize}
  \item [(1)] Lower bound estimate of the major term.
  \begin{equation}\aligned\nonumber
 y^{\frac{1}{2}} M(\alpha;y)-(\frac{\sqrt{3}}{2})^{\frac{1}{2}}M(\alpha;\frac{\sqrt{3}}{2})\geq 8\pi{\alpha}^{\frac{3}{2}}e^{-\frac{\pi\alpha}{y}}>0.
\endaligned\end{equation}
\item [(2)] An auxiliary estimate about the error term.
   \begin{equation}\aligned\nonumber
 \left| y^{\frac{1}{2}}{\alpha}^{-\frac{3}{2}}  e^{\frac{\pi\alpha}{y}}\mathcal{E}(\alpha;y)\right|+\left| (\frac{\sqrt{3}}{2})^{\frac{1}{2}} {\alpha}^{-\frac{3}{2}}e^{\frac{\pi\alpha}{2}}\mathcal{E}(\alpha;\frac{\sqrt{3}}{2})\right|\leq 10^{-5}.
  \endaligned\end{equation}
    \item [(3)] Comparison of the error term with the major term.
  \begin{equation}\aligned\nonumber
\left|\frac{ y^{\frac{1}{2}}\mathcal{E}(\alpha;y)-(\frac{\sqrt{3}}{2})^{\frac{1}{2}}\mathcal{E}(\alpha;\frac{\sqrt{3}}{2})}{y^{\frac{1}{2}}M(\alpha;y)
-(\frac{\sqrt{3}}{2})^{\frac{1}{2}}M(\alpha;\frac{\sqrt{3}}{2})}\right|\leq 10^{-5}.
\endaligned\end{equation}
\item [(4)]
$y^{\frac{1}{2}}\big(M(\alpha;y)+ \mathcal{E}(\alpha;y)\big)-(\frac{\sqrt{3}}{2})^{\frac{1}{2}}\big(M(\alpha;\frac{\sqrt{3}}{2})+ \mathcal{E}(\alpha;\frac{\sqrt{3}}{2})\big)\geq 25{\alpha}^{\frac{3}{2}}e^{-\frac{\pi\alpha}{y}}(1-10^{-5})>0.$
      \end{itemize}

\end{lemma}
\begin{proof}
Note that $\frac{\alpha}{y}\geq 1$. By Lemmas \ref{4.3lemma3} and \ref{lemma4.20}, one has
 \begin{equation}\aligned\nonumber
\sqrt{y} M(\alpha;y)-(\frac{\sqrt{3}}{2})^{\frac{1}{2}}M(\alpha;\frac{\sqrt{3}}{2})\geq& 10\pi{\alpha}^{\frac{5}{2}}y^{-1}e^{-\frac{\pi\alpha}{y}}-
\frac{396\sqrt{2}}{25}{\alpha}^{\frac{3}{2}}e^{-\frac{\sqrt{3}}{2}\pi\alpha}\big( 1+\frac{\sqrt{3}\pi\alpha}{11}\big)\\
\geq& 10\pi{\alpha}^{\frac{3}{2}}e^{-\frac{\pi\alpha}{y}}\big(1-\frac{198\sqrt{2}}{ 125\pi}e^{-\pi\alpha(\frac{\sqrt{3}}{2}-\frac{1}{y})}(1+\frac{\sqrt{3}\pi\alpha}{11})\big).\\
\geq&8\pi{\alpha}^{\frac{3}{2}}e^{-\frac{\pi\alpha}{y}}.
\endaligned\end{equation}
which yields item (1). By Lemma \ref{lem4.23} and \eqref{Poisson}, we have
 \begin{equation}\aligned\label{4.3lemma2error}
 &\left| y^{\frac{1}{2}}{\alpha}^{-\frac{3}{2}}  e^{\frac{\pi\alpha}{y}}\mathcal{E}(\alpha;y)\right|\leq 4\pi(1+4\pi\alpha y)e^{-\pi\alpha(4y-\frac{1}{y})},\\
 &\left| (\frac{\sqrt{3}}{2})^{\frac{1}{2}} {\alpha}^{-\frac{3}{2}}e^{\frac{\pi\alpha}{2}}\mathcal{E}(\alpha;\frac{\sqrt{3}}{2})\right|\leq (\frac{\sqrt{3}}{2})^{\frac{1}{2}}\cdot 4\pi (1+2\sqrt{3}\pi\alpha)e^{-(2\sqrt{3}-\frac{1}{2})\pi\alpha}.
  \endaligned\end{equation}
Item (2) follows by \eqref{4.3lemma2error} and the conditions that $\alpha\geq 2, y\geq2.$ Combining item (1) with $y\geq2$, one has
\begin{equation}\aligned\label{4.3lemma3eq1}
\left|\frac{ y^{\frac{1}{2}}\mathcal{E}(\alpha;y)-(\frac{\sqrt{3}}{2})^{\frac{1}{2}}\mathcal{E}(\alpha;\frac{\sqrt{3}}{2})}{y^{\frac{1}{2}}M(\alpha;y)
-(\frac{\sqrt{3}}{2})^{\frac{1}{2}}M(\alpha;\frac{\sqrt{3}}{2})}\right|\leq&\left|\frac{ y^{\frac{1}{2}}\mathcal{E}(\alpha;y)-(\frac{\sqrt{3}}{2})^{\frac{1}{2}}\mathcal{E}(\alpha;\frac{\sqrt{3}}{2})}{8\pi{\alpha}^{\frac{3}{2}}e^{-\frac{\pi\alpha}{y}}}\right|\\
\leq&\frac{1}{8\pi}\bigg( \left|y^{\frac{1}{2}} \alpha^{-\frac{3}{2}}e^{\frac{\pi\alpha}{y}}\mathcal{E}(\alpha;y) \right| +\left|(\frac{\sqrt{3}}{2})^{\frac{1}{2}} \alpha^{-\frac{3}{2}}e^{\frac{\pi\alpha}{y}} \mathcal{E}(\alpha;\frac{\sqrt{3}}{2})\right| \bigg)\\
\leq&\frac{1}{8\pi}\bigg( \left|y^{\frac{1}{2}} \alpha^{-\frac{3}{2}}e^{\frac{\pi\alpha}{y}}\mathcal{E}(\alpha;y) \right| +\left|(\frac{\sqrt{3}}{2})^{\frac{1}{2}} \alpha^{-\frac{3}{2}}e^{\frac{\pi\alpha}{2}} \mathcal{E}(\alpha;\frac{\sqrt{3}}{2})\right| \bigg).\\
\endaligned\end{equation}
Then, \eqref{4.3lemma3eq1} and item (2) together yield item (3). Items (1) and (3) yield item (4).
\end{proof}

 For simplicity, we denote that
 $
\delta(y):=\underset{n=2}{\overset{\infty}{\sum}}e^{-\pi({n}^2-1)y}.
$

The following lemma provides the bounds used in Lemma \ref{4.3lemma2}.
\begin{lemma} \label{4.3lemma3} Assume that $ \alpha\geq y\geq 2$. Then
\begin{itemize}
  \item [(1)]
  $M_{1}(\alpha;y)\geq 4\sqrt{2}{\alpha}^{\frac{3}{2}}y^{-\frac{1}{2}}\big(e^{-\pi\frac{\alpha}{y}}-e^{-2\pi\frac{\alpha}{y}}(1+\delta(2))\big);$
  \item [(2)]
  $M_{2}(\alpha;y)\geq 4\sqrt{2}{\alpha}^{\frac{3}{2}}{y}^{-\frac{1}{2}}\big(e^{-2\pi \frac{\alpha}{y}}-e^{-\pi \frac{\alpha}{y}}(1+\delta(1))\big)
+8\sqrt{2}\pi{\alpha}^{\frac{5}{2}}{y}^{-\frac{3}{2}}\big(e^{-\pi \frac{\alpha}{y}}-2e^{-2\pi \frac{\alpha}{y}}(1+\mu(2))\big).$
  \item [(3)]$M_{3}(\alpha;y)>0;$
  \item [(4)]  $M_{4}(\alpha;y)>0.$
    \item [(5)] $M(\alpha;y)\geq
  10\pi{\alpha}^{\frac{5}{2}}y^{-\frac{3}{2}}e^{-\frac{\pi\alpha}{y}}.$

       \end{itemize}
       Numerically, $\delta(1)\leq9\cdot10^{-5}$, $\delta(2)\leq7\cdot10^{-9},$ $\mu(2)\leq3\cdot10^{-8}.$
  \end{lemma}
\begin{proof}
 We provide the bounds before the proof. Noting that $\frac{\alpha}{y}\geq1$, we calculate directly that
\begin{equation}\aligned\label{note1}
\underset{n=1}{\overset{\infty}{\sum}}e^{-2\pi{n}^2\frac{\alpha}{y}}&=e^{-2\pi\frac{\alpha}{y}}(1+\underset{n=2}{\overset{\infty}{\sum}}e^{-2\pi({n}^2-1)\frac{\alpha}{y}})\leq e^{-2\pi\frac{\alpha}{y}}(1+\delta(2)),\\
\underset{n=1}{\overset{\infty}{\sum}}e^{-\pi {n}^2\frac{\alpha}{y}}&\leq
e^{-\pi \frac{\alpha}{y}}(1+\delta(1)),\:\:\:\:
\underset{n=1}{\overset{\infty}{\sum}}n^2e^{-2\pi {n}^2\frac{\alpha}{y}}\leq e^{-2\pi \frac{\alpha}{y}}(1+\mu(2)),
\endaligned\end{equation}
and
\begin{equation}\aligned\label{note3}
&\underset{n=2}{\overset{\infty}{\sum}}e^{-2\pi(n-1)n\cdot\frac{\alpha}{y}}\leq \underset{n=2}{\overset{\infty}{\sum}}e^{-2\pi(n-1)n}\leq  10^{-5},\;\;\frac{y}{\alpha}\underset{n=1}{\overset{\infty}{\sum}} e^{-\pi(n-\frac{1}{2})^2\frac{\alpha}{y}}\leq e^{-\frac{\pi\alpha}{4y}}(1+\underset{n=2}{\overset{\infty}{\sum}} e^{-\pi(n-1)n}),\\
&e^{-\pi\alpha y}\underset{n=1}{\overset{\infty}{\sum}} (n-\frac{1}{2})^2e^{-2\pi(n-\frac{1}{2})^2\frac{\alpha}{y}}\leq\frac{1}{4}e^{-\pi\alpha(y+\frac{1}{2y})}(1+\underset{n=2}{\overset{\infty}{\sum}}4(n-\frac{1}{2})^2 e^{-2\pi(n-1)n}).
\endaligned\end{equation}
(1). For $ M_{1}(\alpha;y)$, by \eqref{Poisson}, \eqref{4eq4} and \eqref{note1}, one has
\begin{equation}\aligned\nonumber
 M_{1}(\alpha;y)=4\sqrt{2}{\alpha}^{\frac{3}{2}}y^{-\frac{1}{2}}(\underset{n=1}{\overset{\infty}{\sum}}e^{-\pi{n}^2\frac{\alpha}{y}}-\underset{n=1}{\overset{\infty}{\sum}}e^{-2\pi{n}^2\frac{\alpha}{y}})
\geq 4\sqrt{2}{\alpha}^{\frac{3}{2}}y^{-\frac{1}{2}}\big(e^{-\pi\frac{\alpha}{y}}-e^{-2\pi\frac{\alpha}{y}}(1+\delta(2))\big).
\endaligned\end{equation}
(2). For $M_{2}(\alpha;y)$, by \eqref{4eq4} and \eqref{lem4.22eq1}, one has
\begin{equation}\aligned\label{item2}
M_{2}(\alpha;y)=&4\sqrt{2}{\alpha}^{\frac{3}{2}}{y}^{-\frac{1}{2}}(\underset{n=1}{\overset{\infty}{\sum}}e^{-2\pi {n}^2\frac{\alpha}{y}}-\underset{n=1}{\overset{\infty}{\sum}}e^{-\pi {n}^2\frac{\alpha}{y}})\\
&+8\sqrt{2}\pi{\alpha}^{\frac{5}{2}}{y}^{-\frac{3}{2}}(\underset{n=1}{\overset{\infty}{\sum}}n^2e^{-\pi {n}^2\frac{\alpha}{y}}-2\underset{n=1}{\overset{\infty}{\sum}}n^2e^{-2\pi {n}^2\frac{\alpha}{y}}).\\
\endaligned\end{equation}
which yields item (2).

(3). For $ M_{3}(\alpha;y)$, by \eqref{Poisson} and \eqref{4eq4}, one has
\begin{equation}\aligned\label{lem4.24eq1}
 M_{3}(\alpha;y)= &8\sqrt{2}{\alpha}^{\frac{3}{2}}y^{-\frac{1}{2}} e^{-\pi\alpha(y+\frac{1}{4y})}\big(  (1+2\pi\alpha y)(1+\underset{n=2}{\overset{\infty}{\sum}}e^{-\pi(n-1)n\cdot\frac{\alpha}{y}})\\
 & \;\;\;\;-(1+4\pi\alpha y) e^{-\pi\alpha(y+\frac{1}{4y})}(1+\underset{n=2}{\overset{\infty}{\sum}}e^{-2\pi(n-1)n\cdot\frac{\alpha}{y}})\big).\\
 \endaligned\end{equation}
Therefore, given that $\alpha\geq y\geq 2$, \eqref{lem4.24eq1} and \eqref{note3} yield item (3).\\
(4). For $M_{4}(\alpha;y)$, by \eqref{4eq4} and \eqref{lem4.22eq1}, one has
\begin{equation}\aligned\label{lem4.31eq1}
M_{4}(\alpha;y)=&16\sqrt{2}\pi {\alpha}^{\frac{5}{2}}y^{-\frac{3}{2}}e^{-\pi\alpha y}\underset{n=1}{\overset{\infty}{\sum}} \Big( \big((n-\frac{1}{2})^2-\frac{1}{2\pi}\frac{y}{\alpha}\big)e^{\pi(n-\frac{1}{2})^2\frac{\alpha}{y}}
-2e^{-\pi \alpha y}(n-\frac{1}{2})^2
\Big)e^{-2\pi(n-\frac{1}{2})^2\frac{\alpha}{y}}\\
&+8\sqrt{2}{\alpha}^{\frac{3}{2}}y^{-\frac{1}{2}}e^{-2\pi\alpha y}\underset{n=1}{\overset{\infty}{\sum}} e^{-2\pi(n-\frac{1}{2})^2\frac{\alpha}{y}}.\\
\endaligned\end{equation}

Since $\alpha\geq y\geq 2$, trivially, we have
\begin{equation}\aligned\label{lem4.31eq1bb} \big((n-\frac{1}{2})^2-\frac{1}{2\pi}\frac{y}{\alpha}\big)e^{\pi(n-\frac{1}{2})^2\frac{\alpha}{y}}
-2e^{-\pi \alpha y}(n-\frac{1}{2})^2
>0.
\endaligned\end{equation}
\eqref{lem4.31eq1} and \eqref{lem4.31eq1bb} give that $M_{4}(\alpha;y)>0.$\\
(5). By \eqref{4eq3} and items (1)-(4), one has
\begin{equation}\aligned\nonumber
  M(\alpha;y)\geq &8\sqrt{2}\pi{\alpha}^{\frac{5}{2}}y^{-\frac{3}{2}}e^{-\frac{\pi\alpha}{y}}(1-\frac{y}{2\pi\alpha}\delta(1)-2(1+\mu(2))e^{-\frac{\pi\alpha}{y}}
  -\frac{1}{2\pi}\cdot\delta(2)\cdot\frac{y}{\alpha}e^{-\frac{\pi\alpha}{y}})\\
  \geq & 8\sqrt{2}\pi{\alpha}^{\frac{5}{2}}y^{-\frac{3}{2}}e^{-\frac{\pi\alpha}{y}}(1-\frac{1}{2\pi}\delta(1)-2(1+\mu(2))e^{-\pi}
  -\frac{1}{2\pi}\cdot \delta(2)e^{-\pi})>0.\\
  \endaligned\end{equation}
(6). Item (6) follows by Lemma \ref{lemma4.20}.
\end{proof}

\section{Proof of Theorem \ref{Th1}}

We start with a nonexistence result.
\begin{proposition} \label{Prop51} Assume that $\alpha\geq2$, Then for $b\geq 2\sqrt{2}$, $\argmin_{z\in\mathbb{H}}(\mathcal{K}(\alpha;z)-b \mathcal{K}(2\alpha;z))$ does not exist.
\end{proposition}
\begin{proof}
By Lemma \ref{3Lemma2}, one has
\begin{equation}\aligned\label{5eq1}
 \mathcal{K}(\alpha;z)-b \mathcal{K}(2\alpha;z)
 =&\frac{1}{\pi}{2}^{-\frac{5}{2}}\alpha^{-\frac{5}{2}}y^{\frac{1}{2}}\big((2\sqrt{2}-b)\alpha+2be^{-\pi\frac{y}{2\alpha}}(\pi y-\alpha) \\
  &+4\sqrt{2}e^{-\pi\frac{y}{\alpha}}(\alpha-2\pi y)+o(1)\big),\:\:\:\:\mathrm{as}\:\: y\rightarrow +\infty.
  \endaligned\end{equation}
When $b>2\sqrt{2},$ by \eqref{5eq1},
$ \mathcal{K}(\alpha;z)-b \mathcal{K}(2\alpha;z)\rightarrow -\infty$, as $y\rightarrow +\infty$, hence the minimizer does not exist. When $b=2\sqrt{2}$, $ \mathcal{K}(\alpha;z)-b \mathcal{K}(2\alpha;z)\rightarrow 0^{+}$, as $y\rightarrow +\infty$, therefore Theorem \ref{3Thm1} and Lemma \ref{lem5.1} yield the result.
\end{proof}

We observe that
\begin{lemma}\label{lem5.1} Assume that $ \alpha\geq 2$. Then for $z=\frac{1}{2}+iy\in \Gamma_{c}$, it holds that
\begin{equation}\aligned\nonumber
 \mathcal{K}(\alpha;\frac{1}{2}+iy)-2\sqrt{2} \mathcal{K}(2\alpha;\frac{1}{2}+iy)>0.
  \endaligned\end{equation}
  \end{lemma}
Recalling Lemma \ref{3Lemma2}, for $\alpha\geq2,\:y\geq\frac{\sqrt{3}}{2}$, one has
\begin{equation}\aligned\label{lem5.1denote}
 \mathcal{K}(\alpha;\frac{1}{2}+iy)-2\sqrt{2} \mathcal{K}(2\alpha;\frac{1}{2}+iy)
 =\frac{1}{\pi}{2}^{-\frac{5}{2}}\alpha^{-\frac{5}{2}}y^{\frac{1}{2}}\big(\mathcal{P}(\alpha;y)+\tilde{\mathcal{E}}(\alpha;y)\big),
  \endaligned\end{equation}
  where
  \begin{equation}\aligned\nonumber
 \mathcal{P}(\alpha;y):=&2\sqrt{2} \alpha\sum_{|n|\leq1} e^{-\pi\alpha y n^2}\vartheta(\frac{y}{\alpha};\frac{n}{2})
-2\sqrt{2} \alpha \sum_{|n|\leq1} e^{-2\pi\alpha  y n^2}\vartheta(\frac{y}{2 \alpha};\frac{n}{2})\\
&+4\sqrt{2}\pi \alpha^2y\sum_{|n|\leq1} n^2e^{-\pi\alpha y n^2}\vartheta(\frac{y}{\alpha};\frac{n}{2})
-8\sqrt{2} \pi  \alpha^2y\sum_{|n|\leq1} n^2e^{-2\pi\alpha  y n^2}\vartheta(\frac{y}{2 \alpha};\frac{n}{2})\\
&+4\sqrt{2}y\sum_{|n|\leq1} e^{-\pi\alpha y n^2}\vartheta_X(\frac{y}{\alpha};\frac{n}{2})
-2\sqrt{2} y\sum_{|n|\leq1} e^{-2\pi\alpha y n^2}\vartheta_X(\frac{y}{2 \alpha};\frac{n}{2}),
  \endaligned\end{equation}
  and
 \begin{equation}\aligned\nonumber
\tilde{\mathcal{E}}(\alpha;y):=&2\sqrt{2} \alpha\sum_{|n|\geq2} e^{-\pi\alpha y n^2}\vartheta(\frac{y}{\alpha};\frac{n}{2})
-2\sqrt{2} \alpha \sum_{|n|\geq2} e^{-2\pi\alpha  y n^2}\vartheta(\frac{y}{2 \alpha};\frac{n}{2})\\
&+4\sqrt{2}\pi \alpha^2y\sum_{|n|\geq2} n^2e^{-\pi\alpha y n^2}\vartheta(\frac{y}{\alpha};\frac{n}{2})
-8\sqrt{2} \pi  \alpha^2y\sum_{|n|\geq2} n^2e^{-2\pi\alpha  y n^2}\vartheta(\frac{y}{2 \alpha};\frac{n}{2})\\
&+4\sqrt{2}y\sum_{|n|\geq2} e^{-\pi\alpha y n^2}\vartheta_X(\frac{y}{\alpha};\frac{n}{2})
-2\sqrt{2} y\sum_{|n|\geq2} e^{-2\pi\alpha y n^2}\vartheta_X(\frac{y}{2 \alpha};\frac{n}{2}).
  \endaligned\end{equation}

  To prove Lemma \ref{lem5.1}, we will seek an appropriate lower bound for $\mathcal{P}(\alpha;y)$ and an upper bound for $\tilde{\mathcal{E}}(\alpha;y)$. To state the proof clearly, we decompose $\mathcal{P}(\alpha;y)$ into several parts. We denote that
  \begin{equation}\aligned\nonumber
 \mathcal{P}_{1}(\alpha;y):=&2\sqrt{2} \alpha\vartheta(\frac{y}{\alpha};0)
-2\sqrt{2} \alpha \vartheta(\frac{y}{2 \alpha};0)
+4\sqrt{2}y\vartheta_X(\frac{y}{\alpha};0)
-2\sqrt{2} y\vartheta_X(\frac{y}{2 \alpha};0),\\
 \mathcal{P}_{2}(\alpha;y):=&4\sqrt{2} \alpha(1+2\pi\alpha y)  e^{-\pi\alpha y }\vartheta(\frac{y}{\alpha};\frac{1}{2})
 -4\sqrt{2} \alpha  e^{-2\pi\alpha  y }(1+4\pi\alpha y)\vartheta(\frac{y}{2 \alpha};\frac{1}{2}),\\
\mathcal{P}_{3}(\alpha;y):=&8\sqrt{2}y e^{-\pi\alpha y }\vartheta_X(\frac{y}{\alpha};\frac{1}{2})
-4\sqrt{2} y e^{-2\pi\alpha y }\vartheta_X(\frac{y}{2 \alpha};\frac{1}{2}).
  \endaligned\end{equation}
   Then
   \begin{equation}\aligned\label{lem5.1P}
 \mathcal{P}(\alpha;y)=\mathcal{P}_{1}(\alpha;y)+\mathcal{P}_{2}(\alpha;y)+\mathcal{P}_{3}(\alpha;y).
  \endaligned\end{equation}

Lemma \ref{lem5.1} is implied by \eqref{lem5.1denote} and the following lemma.

  \begin{lemma}\label{lem5.2} Assume that $ \alpha\geq 2$ and $y\geq\frac{\sqrt{3}}{2}$.
   \begin{itemize}
  \item [(1)]  If $\frac{y}{\alpha}\geq1,$ then
   \begin{itemize}
  \item [(1)]  $\mathcal{P}(\alpha;y)\geq 4\sqrt{2}\alpha e^{-\pi\frac{y}{2\alpha}}\big(\frac{\pi y}{\alpha}-1 -( \frac{2\pi y}{\alpha}-1) e^{-\pi \frac{y}{2\alpha}}\big).$
      \item [(2)] $|\tilde{\mathcal{E}}(\alpha;y)|\leq(2\pi\alpha+112\pi {\alpha}^2 y+8 \sqrt{2} y)e^{-4\pi\alpha y}.$
      \end{itemize}
    \item [(2)] If $\frac{y}{\alpha}\in(0,1)$, then
     \begin{itemize}
  \item [(1)]  $ \mathcal{P}(\alpha;y)\geq 8\sqrt{2}\pi{\alpha}^{\frac{5}{2}}y^{-\frac{3}{2}}(e^{-\frac{\pi\alpha}{y}}-2\sqrt{2}e^{-\frac{2\pi\alpha}{y}}).$
    \item [(2)] $|\tilde{\mathcal{E}}(\alpha;y)|\leq 4\pi {\alpha}^{\frac{3}{2}}y^{-\frac{1}{2}}(1+16\alpha y)e^{-4\pi\alpha y}.$
      \end{itemize}
     \item [(3)] $\mathcal{P}(\alpha;y)+\tilde{\mathcal{E}}(\alpha;y)>0.$
      \end{itemize}
   \end{lemma}
  \begin{proof}
  (1). Combining \eqref{TXY} and $\frac{y}{\alpha}\geq1$, one has
 \begin{equation}\aligned\label{lem5.2eq1}
 \mathcal{P}_{1}(\alpha;y)
 =&4\sqrt{2}\alpha\Big(e^{-\pi\frac{y}{2\alpha}}\big(\frac{\pi y}{\alpha}-1 -( \frac{2\pi y}{\alpha}-1) e^{-\pi \frac{y}{2\alpha}}\big)\\
 &\;\;+ \sum_{n=2}^{\infty}e^{-\pi{n}^2\frac{y}{2\alpha}}\big( \frac{\pi y}{\alpha}n^2-1- ( \frac{2\pi y}{\alpha}n^2-1) e^{-\pi{n}^2\frac{y}{2\alpha}}\big)\Big)\\
 \geq& 4\sqrt{2}\alpha e^{-\pi\frac{y}{2\alpha}}\big(\frac{\pi y}{\alpha}-1 -( \frac{2\pi y}{\alpha}-1) e^{-\pi \frac{y}{2\alpha}}\big)>0.\\
  \endaligned\end{equation}
  For $\mathcal{P}_{2}(\alpha;y)$ and $\mathcal{P}_{3}(\alpha;y)$,  similar to the analysis of positiveness of $M_{3}(\alpha;y)$ and $M_{4}(\alpha;y)$ in Lemma \ref{lemma4.19},
 one can get $\mathcal{P}_{2}(\alpha;y)>0$ and $\mathcal{P}_{3}(\alpha;y)>0$. Then, by \eqref{lem5.1P} and \eqref{lem5.2eq1}, we obtain the lower bound estimate of $\mathcal{P}(\alpha;y)$. Similar to the proof of item (2) of Lemma \ref{4.2lemma2}, using \eqref{TXY}, \eqref{Poisson} and Lemma \ref{lem4.23}, we can get an upper bound estimate of $\tilde{\mathcal{E}}(\alpha;y)$.

  (2). As $\frac{\alpha}{y}\geq1$, for $\mathcal{P}_{1}(\alpha;y)$, by \eqref{Poisson}, one has
  \begin{equation}\aligned\label{lem5.3eq1}
 \mathcal{P}_{1}(\alpha;y)=&8\sqrt{2}\pi{\alpha}^{\frac{5}{2}}y^{-\frac{3}{2}}\big(\sum_{n=1}^{\infty}n^2 e^{-\pi n^2 \frac{\alpha}{y}}-2\sqrt{2}\sum_{n=1}^{\infty}n^2 e^{-2\pi n^2 \frac{\alpha}{y}}\big)\\
 =&8\sqrt{2}\pi{\alpha}^{\frac{5}{2}}y^{-\frac{3}{2}}\big( e^{-\pi  \frac{\alpha}{y}}-2\sqrt{2}e^{-2\pi  \frac{\alpha}{y}}+ \sum_{n=2}^{\infty}n^2 e^{-\pi n^2 \frac{\alpha}{y}}(1-2\sqrt{2} e^{-\pi n^2 \frac{\alpha}{y}} )\big)\\
  \geq& 8\sqrt{2}\pi{\alpha}^{\frac{5}{2}}y^{-\frac{3}{2}}(e^{-\frac{\pi\alpha}{y}}-2\sqrt{2}e^{-\frac{2\pi\alpha}{y}}).
  \endaligned\end{equation}
 Similar to the analysis of the positiveness of $M_{3}(\alpha;y)$ and $M_{4}(\alpha;y)$ in Lemma \ref{4.3lemma3}, one gets $\mathcal{P}_{2}(\alpha;y)>0$ and $\mathcal{P}_{3}(\alpha;y)>0.$ Thus, by \eqref{lem5.1P}, \eqref{lem5.3eq1} and the positiveness of $\mathcal{P}_{2}(\alpha;y)$ and $\mathcal{P}_{3}(\alpha;y)$, we get the estimate for $\mathcal{P}(\alpha;y)$. Similar to the proof of item (2) of Lemma \ref{4.2lemma2}, we get the upper bound estimate of $\mathcal{E}(\alpha;y)$. Items (1) and (2) yield item (3).
   \end{proof}

With the previous preparation, we are ready to prove our main theorem.

Case 1: By Theorems \ref{3Thm1} and \ref{4Thm1}, we obtain that, up to the action by the modular group,
\begin{equation}\aligned\label{Lemma511}
\underset{z\in  \mathbb{H}}{\arg \min}\;\big(\mathcal{K}(\alpha;z)-b\mathcal{K}(2\alpha;z)\big)=e^{i\frac{\pi}{3}}
\;\;\hbox{for}\;\;\alpha\geq2\;\hbox{and}\;b\leq2.
\endaligned\end{equation}
 Therefore, we introduce and define that
$$b_{c_{1}}:=\{\max\; b\;\big|\;\underset{z\in  \Gamma_c}{\arg \min}\;\big(\mathcal{K}(\alpha;z)-b\mathcal{K}(2\alpha;z)\big)=e^{i\frac{\pi}{3}}
\;\;\hbox{for}\;\;\alpha\geq2
\}.$$
 Then by Proposition \ref{Prop51} and \eqref{Lemma511}, $b_{c_1}\in(2,2\sqrt2)$. The threshold
 $b_{c_1}$ varies with $\alpha$. For example, $b_{c_{1}}\approx2.8061$ when $\alpha=2$,\; and $b_{c_{1}}\approx2.8242$ when $\alpha=2.5.$

Case 2: $b\in (b_{c_{1}},b_{c_2}).$ When $b<b_{c_2}=2\sqrt{2}$, by Theorem \ref{3Thm1} and \eqref{5eq1}, the minimizer always exists.
Furthermore, by Theorem \ref{3Thm1} and the definition of $b_{c_{1}}$ in \eqref{define}, the minimizer is $\frac{1}{2}+iy_b$ $(y_b>\frac{\sqrt{3}}{2})$, corresponding the skinny-rhombic lattice. A numerical simulation can be found in Figure \ref{excel}. Furthermore, the monotonicity result in Luo-Wei-Zou \cite{LW2021} implies that when $b$ approaches $2\sqrt{2}$, $y_{b}$ approaches $+\infty$.

Case 3: $b\geq b_{c_2}=2\sqrt{2}.$ It follows by Proposition \ref{Prop51}.

\begin{figure}
\centering
 \includegraphics[scale=0.42]{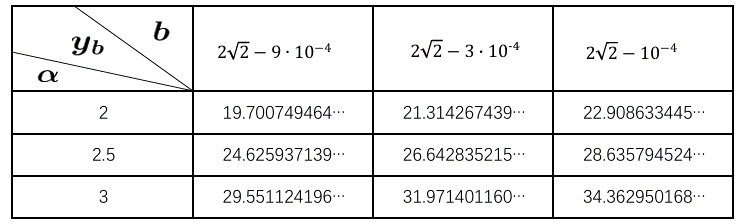}
 \caption{The minimizers $z=\frac{1}{2}+i y_{b}$ corresponding to skinny-rhombic lattices.}
 \label{excel}
\end{figure}

\vskip0.1in

%\noindent
{\bf Acknowledgements.}
The research of S. Luo is partially supported by the National Natural Science Foundation of China (NSFC) under Grant Nos. 12261045 and 12001253, and by the Jiangxi Jieqing Fund under Grant No. 20242BAB23001.

%%%%%%%%%%%%%%%%%%%%%%%%%%%%%%%%%%%%%%%%%%%%%%%%%%%%%%%%%%%%%%%%%%%%%%%%%%%%%%%%%%%%%%%%%%%%%%%%%%%%%%%%%%%%%%%%%%%%%%%%%%%%%%%%%%%%%%

\end{document}